\documentclass[12pt]{amsart}

\usepackage[utf8]{inputenc}
\usepackage{amssymb,amscd,amsmath}
\usepackage{fullpage}
\usepackage{tikz-cd}
\usetikzlibrary{intersections}

\makeatletter
\@namedef{subjclassname@2020}{%
  \textup{2020} Mathematics Subject Classification}
\makeatother

\usepackage{amsthm}

\theoremstyle{plain}
\newtheorem{theorem}{Theorem}[section]
\newtheorem{lemma}[theorem]{Lemma}
\newtheorem{proposition}[theorem]{Proposition}
\newtheorem{corollary}[theorem]{Corollary}
\newtheorem{conjecture}[theorem]{Conjecture}
\newtheorem{example}[theorem]{Example}

\theoremstyle{definition}
\newtheorem{definition}[theorem]{Definition}

\theoremstyle{remark}
\newtheorem{remark}[theorem]{Remark}

\newcommand{\C}{\mathbf{C}}
\newcommand{\Z}{\mathbf{Z}}
\newcommand{\Q}{\mathbf{Q}}
\newcommand{\R}{\mathbf{R}}
\newcommand{\PP}{\mathbf{P}}
\newcommand{\h}{\mathcal{H}}	
\newcommand{\M}{\mathcal{M}}
\newcommand{\LL}{\widehat{\mathcal{L}}}
\newcommand{\Hom}{\mathrm{Hom}}
\newcommand{\Aut}{\mathrm{Aut}}
\newcommand{\dv}{\operatorname{div}}	
\newcommand{\dlog}{\operatorname{dlog}}	
\newcommand{\darg}{\operatorname{darg}}	
\newcommand{\Eis}{\mathrm{Eis}}	
\newcommand{\Res}{\operatorname{Res}}	
\newcommand{\GL}{\mathrm{GL}}
\newcommand{\SL}{\mathrm{SL}}
\newcommand{\Spec}{\operatorname{Spec}}
\DeclareMathOperator{\ord}{ord}
\DeclareMathOperator{\Gal}{Gal}
\renewcommand{\Re}{\text{Re}\,}
\renewcommand{\Im}{\text{Im}\,}

\newcommand{\sabcd}[4]{\left(\begin{smallmatrix}#1&#2\\#3& #4\end{smallmatrix}\right)}
\newcommand{\bolda}{{\boldsymbol a}}
\newcommand{\boldb}{{\boldsymbol b}}
\newcommand{\boldc}{{\boldsymbol c}}
\newcommand{\boldd}{{\boldsymbol d}}
\newcommand{\boldu}{{\boldsymbol u}}
\newcommand{\boldv}{{\boldsymbol v}}
\newcommand{\boldw}{{\boldsymbol w}}
\newcommand{\boldx}{{\boldsymbol x}}
\newcommand{\boldy}{{\boldsymbol y}}
\newcommand{\boldz}{{\boldsymbol z}}
\newcommand{\boldt}{{\boldsymbol t}}
\newcommand{\boldzero}{{\boldsymbol 0}}

	\title{On the $K_4$ group of modular curves}
	\author{Fran\c{c}ois Brunault}
	\date{\today}
	\usepackage{hyperref}
	\hypersetup{pdfauthor={Fran\c{c}ois Brunault},%
	            pdftitle={On the $K_4$ group of modular curves},%
	            pdfsubject={Number Theory},%
	            pdfcreator={XeLaTeX},%
	            colorlinks=true,%
	            linkcolor=[rgb]{0,.6,1},
	            urlcolor = blue,
	            citecolor=blue}

\begin{document}

\begin{abstract}
We construct elements in the $K_4$ group of modular curves using the polylogarithmic complexes of weight 3 defined by Goncharov and De Jeu. The construction is uniform in the level and relies on new modular units arising as cross-ratios of division values of the Weierstra\ss{} $\wp$ function. These units provide explicit triangulations of the $3$-term relations in $K_2$, which in turn give rise to elements in $K_4$. Based on numerical computations and on results of Wang, we conjecture that these elements are proportional to the Beilinson elements defined via Eisenstein symbols.
\end{abstract}

\maketitle

\section{Introduction}

Algebraic $K$-theory is a fundamental invariant with deep relations to algebraic geometry and number theory. For example, it enters into the general conjectures of Beilinson on special values of $L$-functions of varieties over number fields \cite{Bei84, DS91, Nek94}. For a field $F$, the $K$-groups $K_0(F)$ and $K_1(F)$ are isomorphic to $\Z$ and $F^\times$, respectively. Moreover, Matsumoto's theorem gives a presentation of $K_2(F)$. But in general $K$-groups are very hard to compute. Goncharov and De Jeu have defined explicit polylogarithmic complexes, which are expected to compute the motivic cohomology of $F$, and thus the rational $K$-groups of $F$ (see Section~\ref{sec: goncharov de jeu}). The results of Goncharov and De Jeu on these complexes suggest that it should be possible to do explicit constructions in the motivic cohomology of algebraic varieties, in particular for varieties of arithmetic interest.

In this article, we investigate the motivic cohomology of the modular curve $Y(N)$ of level $N$ over $\Q$. The groups of interest are $H^2_\M(Y(N),\Q(n))$ for $n \geqslant 2$, which are isomorphic to the $K$-groups $K_{2n-2}^{(n)}(Y(N))$. We expect that there are constructions of modular nature in the relevant polylogarithmic complexes. For $n = 2$, taking the Milnor symbol of two modular units provides the so-called Beilinson--Kato elements, which have important applications in number theory: for example, they are used in Kato's construction of a Euler system for modular forms \cite{Kat04}. For $n \geqslant 3$, Beilinson has also constructed elements in $K_{2n-2}^{(n)}(Y(N))$ using Eisenstein symbols \cite{Bei86}. His construction does not use the polylogarithmic complex. The regulator of these elements can be computed in terms of $L$-functions of modular forms \cite{Bei86, Gea05, Wan20}.

However, the shape of the Beilinson elements is not well-suited for applications to the Mahler measure of certain polynomials (see Section \ref{subsec: mahler}). We present a new construction of explicit elements in the case $n=3$, in other words for the group $K_4^{(3)}(Y(N))$, using the polylogarithmic complex. These elements turn out to be more directly related to the Mahler measure, and this was one motivation for our work. It is of course interesting to try to compare this construction with Beilinson's construction. In this regard, we propose a conjecture, based on numerical computations, that the two kinds of elements are in fact proportional (see Conjecture \ref{main conj} below).

Let us outline briefly the ideas of our construction. The basic ingredients are the Siegel units $g_\boldx$, defined for $\boldx \in (\Z/N\Z)^2$, which are canonical modular units on the modular curve $Y(N)$ (see Definition \ref{def gx}). We prove the following $3$-term relations in the $K_2$ group of the modular curve (Theorem \ref{thm 3-term K2}).

\begin{theorem} \label{main thm 1}
For any $\bolda, \boldb, \boldc \in (\Z/N\Z)^2$ such that $\bolda + \boldb + \boldc = 0$, we have
\begin{equation} \label{eq main thm 1}
\{g_\bolda, g_\boldb\} + \{g_\boldb, g_\boldc\} + \{g_\boldc, g_\bolda\} = 0 \qquad \textrm{in } K_2(Y(N)) \otimes \Z\Bigl[\frac{1}{6N}\Bigr].
\end{equation}
\end{theorem}

This theorem was previously known with $\Q$-coefficients \cite{Bru08, Gon08, BBPS23}. Recently, Sharifi and Venkatesh \cite{SV24} constructed a Hecke equivariant map from the homology of $X_1(N)(\C)$ to $K_2(X_1(N)) \otimes \Z\bigl[\frac{1}{30N}\bigr]$, giving a new perspective on Theorem \ref{main thm 1}.

Our proof of Theorem \ref{main thm 1} differs from \cite{Gon08, BBPS23, SV24} and expands on the approach in \cite{Bru08}. The idea is to find explicit Steinberg relations underlying the relation \eqref{eq main thm 1}. To this end, we introduce in Section \ref{sec: uabcd} new modular units $u(\bolda, \boldb, \boldc, \boldd)$ with $\bolda, \boldb, \boldc, \boldd$ in $(\Z/N\Z)^2$, obtained as cross-ratios of $N$-division values of the Weierstra\ss{} $\wp$ function. A crucial property is that $1-u(\bolda, \boldb, \boldc, \boldd) = u(\bolda, \boldc, \boldb, \boldd)$ is also a modular unit. We then prove (Theorem~\ref{triangulation manin3}):

\begin{theorem} \label{main thm 2}
The relation \eqref{eq main thm 1} has a lift
\begin{equation*}
g_\bolda \wedge g_\boldb + g_\boldb \wedge g_\boldc + g_\boldc \wedge g_\bolda = \sum_i n_i \cdot u_i \wedge (1-u_i) \qquad \textrm{in } \Lambda^2 \mathcal{O}(Y(N))^\times \otimes \Z\Bigl[\frac{1}{6N}\Bigr],
\end{equation*}
where the $u_i$ are explicit modular units of the form $u(\boldx, \boldy, \boldz, \boldt)$, with explicit coefficients $n_i \in \Z\bigl[\frac{1}{2N}\bigr]$.
\end{theorem}

The modular unit $u(\bolda, \boldb, \boldc, \boldd)$ leads to the following additive relation between Siegel functions (see Corollary \ref{cor additive relations}):
\begin{equation} \label{eq additive Siegel}
\zeta \cdot g_{\boldc+\bolda} g_{\boldc-\bolda} g_{\boldd+\boldb} g_{\boldd-\boldb} + \zeta' \cdot g_{\boldb+\bolda} g_{\boldb-\bolda} g_{\boldd+\boldc} g_{\boldd-\boldc} = g_{\boldc+\boldb} g_{\boldc-\boldb} g_{\boldd+\bolda} g_{\boldd-\bolda},
\end{equation}
where $\zeta$ and $\zeta'$ are roots of unity depending on $\bolda, \boldb, \boldc, \boldd$. In our approach, this can be seen as the key identity to prove Theorem \ref{main thm 1}.

Building on Theorem \ref{main thm 2}, we show that the symbol
\begin{equation*}
\tilde{\xi}(\bolda, \boldb) = \sum_i n_i \cdot \{u_i\}_2 \otimes \Bigl( \frac{g_\boldb}{g_\bolda} \Bigr)
\end{equation*}
is a cocycle in the weight $3$ polylogarithmic complex of $Y(N)$ (see Section \ref{sec: construction}). Using results of De Jeu \cite{Jeu96}, this cocycle gives rise to an element $\xi(\bolda, \boldb)$ in $K_4^{(3)}(Y(N))$. We also discuss how to get elements in $K_4^{(3)}(X(N))$, where $X(N)$ is the compactification of $Y(N)$.

We devise in Section \ref{sec: num} a method to compute numerically the image of $\xi(\bolda, \boldb)$ under the Beilinson regulator map. This allows us, in Section \ref{sec: num beilinson}, to check numerically Beilinson's conjecture on $L(E,3)$ for every elliptic curve $E$ over $\Q$ of conductor $N \leqslant 50$. Similar computations lead us to the following conjecture (see Conjecture \ref{conj xi Eis001}).

\begin{conjecture} \label{main conj}
Let $\bolda, \boldb \in (\Z/N\Z)^2$, and let $\Eis^{0,0,1}(\bolda, \boldb)$ be the associated Beilinson element in $K_4^{(3)}(Y(N))$. Then we have
\begin{equation*}
\xi(\bolda, \boldb) = \pm \frac{N^2}{3} \Eis^{0,0,1}(\bolda, \boldb).
\end{equation*}
\end{conjecture}

\subsection{Overview}
Here is a brief outline of the article. We define in Section \ref{sec: uabcd} the modular units $u(\bolda, \boldb, \boldc, \boldd)$. We prove in Section \ref{sec: 3-term} the $3$-term relations in $K_2(Y(N)) \otimes \Z\bigl[\frac{1}{6N}\bigr]$ and their lifts. After recalling the polylogarithmic complexes in Section \ref{sec: goncharov de jeu}, we construct the elements in $K_4^{(3)}(Y(N))$ in Section \ref{sec: construction} and discuss their properties. We compute numerically their regulators in Sections~\ref{sec: num} and~\ref{sec: num beilinson}, leading to the conjectural comparison with the Beilinson elements in Section \ref{sec: comparison}. Finally, we give in Section \ref{sec: perspectives} some perspectives arising from this work. In particular, we mention how the constructions of this article have been used recently in \cite{Bru23, BZ23} to prove a conjecture relating the Mahler measure of the polynomial $(1+x)(1+y)+z$ and $L(E,3)$, where $E$ is an elliptic curve of conductor $15$. Throughout the text, we provide a ``running example'' for the modular curve $X_1(15)$, which is isomorphic to $E$.

\;

\textbf{Acknowledgements.} I am grateful to the referee for pointing out the additive relation~\eqref{eq additive Siegel} between Siegel functions and the link with Fay's trisecant identity (Remark \ref{rk 4var theta}), as well as for the remarks that have enriched the article and improved the exposition.
I would like to thank Emmanuel Lecouturier for asking me for a more natural proof of the 3-term relations in $K_2$ compared to my earlier proof in \cite{Bru08}, and for his help in the last step of the proof of Theorem~\ref{triangulation manin3}.
I am also grateful to Spencer Bloch who asked me for an explicit triangulation of the 3-term relations, to Rob de Jeu for reading carefully my manuscript and the many useful comments, and to Andrea Surroca for interesting discussions around the $S$-unit equation.
Finally, I would like to thank my colleagues for fruitful exchanges, among them Vasily Golyshev, Matt Kerr and the International Groupe de Travail on differential equations in Paris, Matilde Lal\'{\i}n, Riccardo Pengo, Jun Wang, Weijia Wang and Wadim Zudilin.

\section{Modular units and cross-ratios} \label{sec: uabcd}

We denote by $\h = \{\tau \in \C : \Im(\tau)>0\}$ the Poincaré upper half-plane. Let $N > 1$ be an integer. For any vector $\bolda=(a_1,a_2) \in (\Z/N\Z)^2$, $\bolda \neq (0,0)$, we define
\begin{equation*}
\wp_\bolda(\tau) = \wp\Bigl(\frac{\tilde{a}_1 \tau + \tilde{a}_2}{N}, \tau \Bigr) \qquad (\tau \in \h),
\end{equation*}
where $\wp(z,\tau)$ is the Weierstra\ss{} $\wp$ function \cite{CS17, KK25, Lan87}, and $\tilde{a}_1, \tilde{a}_2$ are representatives of $a_1,a_2$ in $\Z$. The function $\wp_\bolda$ is well-defined because $\wp(z, \tau)$ is doubly periodic with respect to the lattice $\Z \tau + \Z$. Since $\wp(z, \tau)$ is an even function of $z$, we also have $\wp_{-\bolda} = \wp_\bolda$, so that $\wp_\bolda$ makes sense for $\bolda \in (\Z/N\Z)^2/\pm 1$, $\bolda \neq (0,0)$.

\begin{definition}
Let $\bolda, \boldb, \boldc, \boldd$ be four pairwise distinct elements of $(\Z/N\Z)^2/\pm 1$. We define the function $u(\bolda,\boldb,\boldc,\boldd)$ on $\h$ as the cross-ratio of the functions $\wp_\bolda,\wp_\boldb,\wp_\boldc,\wp_\boldd$:
\begin{equation*}
u(\bolda,\boldb,\boldc,\boldd) = [\wp_\bolda, \wp_\boldb, \wp_\boldc, \wp_\boldd] = \frac{\wp_\boldc - \wp_\bolda}{\wp_\boldc - \wp_\boldb} \Big/ \frac{\wp_\boldd - \wp_\bolda}{\wp_\boldd - \wp_\boldb}.
\end{equation*}
\end{definition}

If one of the indices $\bolda, \boldb, \boldc, \boldd$ is equal to $\boldzero = (0,0)$, then $\wp_{\boldzero}$ has to be interpreted as $\infty$ (this is consistent with the fact that the $\wp$ function has a pole at the origin). For example, we have
\begin{equation*}
u(\boldzero, \boldb, \boldc, \boldd) = \frac{\wp_\boldd - \wp_\boldb}{\wp_\boldc - \wp_\boldb}.
\end{equation*}

Recall that a modular unit for a congruence subgroup $\Gamma$ of $\SL_2(\Z)$ is a holomorphic function on $\h$ which is non-vanishing, invariant under $\Gamma$ and meromorphic at the cusps of~$\Gamma \backslash \h$. Let $\Gamma(N) = \ker(\SL_2(\Z) \to \SL_2(\Z/N\Z))$ denote the principal congruence subgroup of level $N$.

\begin{lemma} \label{lem uabcd mod unit}
If $\bolda, \boldb, \boldc, \boldd$ are pairwise distinct elements of $(\Z/N\Z)^2/\pm 1$, then the function $u(\bolda, \boldb, \boldc, \boldd)$ is a modular unit for $\Gamma(N)$.
\end{lemma}

\begin{proof}
For a fixed $\tau$ in $\h$, we know from the theory of elliptic functions that $\wp(z, \tau) = \wp(z', \tau)$ if and only if $z' = \pm z$ modulo $\Z \tau + \Z$ \cite[Corollary 2.1.9]{CS17}. Since $\bolda, \boldb, \boldc, \boldd$ are distinct in $(\Z/N\Z)^2/\pm 1$, the function $u(\bolda, \boldb, \boldc, \boldd)$ has no zeros or poles in $\h$. Moreover, the modular behaviour of the $\wp$ function \cite[Section I.4.1, eq.~(4.9)]{KK25} implies
\begin{equation} \label{wpx gamma}
\wp_\boldx(\gamma \tau) = (c\tau+d)^2 \wp_{\boldx \gamma}(\tau) \qquad (\gamma = \sabcd{a}{b}{c}{d} \in \SL_2(\Z)),
\end{equation}
where $\boldx \in (\Z/N\Z)^2 \setminus \{\boldzero\}$ is seen as a row vector. In particular $\wp_\boldx$ is modular of weight~$2$ for $\Gamma(N)$, and in fact it is a modular form \cite[Chapter VII, Theorem 7]{Sch74}. It follows that $u(\bolda, \boldb, \boldc, \boldd)$ is invariant under $\Gamma(N)$ and is meromorphic at the cusps.
\end{proof}

Modular units of the form $(\wp_\bolda - \wp_\boldb)/(\wp_\boldc - \wp_\boldd)$ are called Weierstra\ss{} units in \cite[Chapter~2, Section~6]{KL81}. Here $u(\bolda, \boldb, \boldc, \boldd)$ is a quotient of two Weierstra\ss{} units, but is not a priori a Weierstra\ss{} unit. We note that Bolbachan also considered the cross-ratio of values of the $\wp$ function in relation with the elliptic dilogarithm \cite{Bol22}. The cross-ratio is viewed there as an elliptic function, not as a modular one.

Our next goal is to express the modular units $u(\bolda, \boldb, \boldc, \boldd)$ in terms of the \emph{Siegel functions}, whose definition we now recall \cite[1.9]{Kat04}.

\begin{definition} \label{def gx}
For $\boldx = (x_1, x_2) \in (\Z/N\Z)^2$, $\boldx \neq (0,0)$, the Siegel function $g_\boldx$ is defined by the infinite product
\begin{equation} \label{eq def gx}
g_\boldx(\tau) = q^{\frac12 B_2(\tilde{x}_1/N)} \prod_{n \geqslant 0} (1-q^{n+\tilde{x}_1/N} \zeta_N^{x_2}) \prod_{n \geqslant 1} (1-q^{n-\tilde{x}_1/N} \zeta_N^{-x_2}) \qquad (\tau \in \h),
\end{equation}
where $\tilde{x}_1$ is the representative of $x_1$ satisfying $0 \leqslant \tilde{x}_1 \leqslant N-1$, $B_2(t) = t^2-t+1/6$ is the second Bernoulli polynomial, $q^\alpha = e^{2\pi i \alpha \tau}$ and $\zeta_N = e^{2\pi i/N}$. By convention, we put $g_{\boldzero} = 1$.
\end{definition}

The function $g_\boldx^{12N}$ is a modular unit for $\Gamma(N)$ \cite[Chapter 2, Theorem 1.2]{KL81} (a different normalisation of $g_\boldx$ is used there, but this does not affect the modularity property).

We will need the following theta function
\begin{equation*}
\theta(z,\tau) = \prod_{n=0}^\infty (1- e^{2\pi i(n\tau+z)}) (1-e^{2\pi i((n+1)\tau -z)}) \qquad (z \in \C, \, \tau \in \h).
\end{equation*}
Up to a standard factor, this is the first (odd) Jacobi theta function (see Remark \ref{rk 4var theta}). Note that by definition, the Siegel function $g_\boldx$ is essentially the evaluation of $\theta(z,\tau)$ at a point of $\frac{1}{N}(\Z\tau+\Z)$. The following proposition is crucial in the proof of Theorem \ref{main thm 2}.

\begin{proposition} \label{pro uabcd gx}
Let $\bolda, \boldb, \boldc, \boldd$ be elements of $(\Z/N\Z)^2$ whose images in $(\Z/N\Z)^2/\pm 1$ are pairwise distinct. We have the following relation between functions on $\h$:
\begin{equation} \label{eq uabcd gx}
u(\bolda, \boldb, \boldc, \boldd) = \zeta \cdot \frac{g_{\boldc+\bolda} g_{\boldc-\bolda} g_{\boldd+\boldb} g_{\boldd-\boldb}}{g_{\boldc+\boldb} g_{\boldc-\boldb} g_{\boldd+\bolda} g_{\boldd-\bolda}},
\end{equation}
where $\zeta$ is a $N'$th root of unity depending on $\bolda, \boldb, \boldc, \boldd$, with $N' = N$ if $N$ is even, and $N' = 2N$ if $N$ is odd.
\end{proposition}

\begin{proof}
We first prove such a relation at the level of the theta function. We will fix $\tau \in \h$ and write $\wp(z) = \wp(z,\tau)$ and $\theta(z) = \theta(z,\tau)$. For $x,y$ in $\C \setminus (\Z\tau+\Z)$, we have by \cite[Corollary~I.5.6(a)]{Sil94}
\begin{equation*}
\wp(x)-\wp(y) = - \frac{\sigma(x+y) \sigma(x-y)}{\sigma(x)^2 \sigma(y)^2},
\end{equation*}
where $\sigma(z) = \sigma(z,\tau)$ is the Weierstra\ss{} $\sigma$ function. We deduce
\begin{equation*}
[\wp(x), \wp(y), \wp(z), \wp(t)] = \frac{\sigma(z+x) \sigma(z-x) \sigma(t+y) \sigma(t-y)}{\sigma(z+y) \sigma(z-y) \sigma(t+x) \sigma(t-x)}.
\end{equation*}
Using the product expansion for the $\sigma$ function \cite[Theorem I.6.4]{Sil94}, we get
\begin{equation} \label{eq uxyzt}
[\wp(x), \wp(y), \wp(z), \wp(t)] = \frac{\theta(z+x) \theta(z-x) \theta(t+y) \theta(t-y)}{\theta(z+y) \theta(z-y) \theta(t+x) \theta(t-x)}.
\end{equation}
By continuity, this formula still holds when one of the variables $x,y,z,t$ belongs to $\Z\tau+\Z$. It remains to specialise \eqref{eq uxyzt} at $N$-torsion points of the elliptic curve $\C/(\Z\tau+\Z)$. Let $z = \frac{1}{N}(z_1 \tau + z_2)$ with $(z_1, z_2) \in \Z^2 \setminus N\Z^2$. By \cite[eq.~(4)]{Yan04}, we have
\begin{equation} \label{eq thetax gx}
\theta(z, \tau) = q^{-B_2(z_1/N)/2} (-\zeta_N^{-z_2})^{\lfloor z_1/N \rfloor} g_{\bar{z}_1, \bar{z}_2},
\end{equation}
where $\lfloor \cdot \rfloor$ denotes the integer part. Using \eqref{eq uxyzt} with lifts $x,y,z,t$ in $\C$ of the torsion points of $\C/(\Z\tau+\Z)$ associated to $\bolda, \boldb, \boldc, \boldd$, together with the identity
\begin{align*}
& B_2(X+Z)+B_2(X-Z)+B_2(Y+T)+B_2(Y-T) \\
& \qquad -B_2(Y+Z)-B_2(Y-Z)-B_2(X+T)-B_2(X-T)=0,
\end{align*}
we obtain the proposition.
\end{proof}

\begin{remark}
The root of unity $\zeta$ in \eqref{eq uabcd gx} can be determined from \eqref{eq thetax gx}.
\end{remark}

\begin{remark}
We would like to point out an error in \cite{Bru08}: the first equation on p.~288 is off by a $N'$th root of unity. This root of unity can be determined from \eqref{eq thetax gx}, using the last equation on p.~287 of \emph{op.~cit.} and noting that $\gamma(q,a_1,a_2) = \theta\bigl(\frac{1}{N}(a_1 \tau + a_2), \tau\bigr)$.
\end{remark}

The definition of $u(\bolda, \boldb, \boldc, \boldd)$ as a cross-ratio makes it clear that
\begin{equation} \label{eq uabcd uacbd}
u(\bolda, \boldb, \boldc, \boldd) + u(\bolda, \boldc, \boldb, \boldd) = 1.
\end{equation}
Thanks to Proposition \ref{pro uabcd gx}, this produces non-trivial additive relations between products of Siegel functions.

\begin{corollary} \label{cor additive relations}
Let $\bolda, \boldb, \boldc, \boldd$ be elements of $(\Z/N\Z)^2$ whose images in $(\Z/N\Z)^2/\pm 1$ are pairwise distinct. Then
\begin{equation*}
\zeta \cdot g_{\boldc+\bolda} g_{\boldc-\bolda} g_{\boldd+\boldb} g_{\boldd-\boldb} + \zeta' \cdot g_{\boldb+\bolda} g_{\boldb-\bolda} g_{\boldd+\boldc} g_{\boldd-\boldc} = g_{\boldc+\boldb} g_{\boldc-\boldb} g_{\boldd+\bolda} g_{\boldd-\bolda},
\end{equation*}
where $\zeta$ and $\zeta'$ are $N'$th roots of unity depending on $\bolda, \boldb, \boldc, \boldd$, with $N' = N$ if $N$ is even, and $N' = 2N$ if $N$ is odd.
\end{corollary}

\begin{proof}
This follows from \eqref{eq uabcd gx} and \eqref{eq uabcd uacbd}, using also the elementary relation
\begin{equation} \label{eq g-x}
g_{-x_1, -x_2} = \begin{cases} g_{x_1, x_2} & \textrm{if } x_1 \neq 0, \\
- \zeta_N^{-x_2} g_{x_1, x_2} & \textrm{if } x_1 = 0.
\end{cases} \qedhere
\end{equation}
\end{proof}

\begin{remark} \label{rk 4var theta}
Corollary \ref{cor additive relations} holds at the level of the theta function. Indeed, the equation \eqref{eq uxyzt} implies that for any $x,y,z,t$ in $\C$, we have
\begin{equation*}
\begin{split}
& e^{2\pi i (z-y)} \theta(y+x) \theta(y-x) \theta(t+z) \theta(t-z) + \theta(z+y) \theta(z-y) \theta(t+x) \theta(t-x) \\
& = \theta(z+x) \theta(z-x) \theta(t+y) \theta(t-y),
\end{split}
\end{equation*}
where we used the identity $\theta(-z) = -e^{-2\pi iz} \theta(z)$. In terms of the Jacobi theta function $\vartheta_{11}(z, \tau) = -i e^{\pi i(\tau/6-z)} \eta(\tau) \theta(z,\tau)$, where $\eta(\tau)$ is the Dedekind eta function, this becomes
\begin{equation} \label{eq 4var theta11}
\begin{split}
& \vartheta_{11}(y+x) \vartheta_{11}(y-x) \vartheta_{11}(t+z) \vartheta_{11}(t-z) + \vartheta_{11}(z+y) \vartheta_{11}(z-y) \vartheta_{11}(t+x) \vartheta_{11}(t-x) \\
& = \vartheta_{11}(z+x) \vartheta_{11}(z-x) \vartheta_{11}(t+y) \vartheta_{11}(t-y).
\end{split}
\end{equation}
One can show that \eqref{eq 4var theta11} is a consequence of Fay's trisecant identity for the elliptic curve $\C/(\Z\tau+\Z)$. Namely, taking $\alpha = (\tau+1)/2$ and $z = \alpha - P_1 - P_3$ in the last displayed equation of \cite[p.~354]{FK92}, we recover \eqref{eq 4var theta11}. This gives another proof of Corollary \ref{cor additive relations}.
\end{remark}

We now show that the modular units $u(\bolda, \boldb, \boldc, \boldd)$ are defined over $\Q$. For $N \geqslant 3$, let $Y(N)$ be the modular curve over $\Q$ classifying elliptic curves with level $N$ structure as in \cite[Section 1]{Kat04}. For $\tau \in \h$, we denote by $\nu(\tau) \in Y(N)(\C)$ the isomorphism class of the elliptic curve $\C/(\Z\tau+\Z)$ with the level structure $(\tau/N, 1/N)$. The group $\GL_2(\Z/N\Z)$ acts from the left on $Y(N)$ by the rule $\gamma \cdot (E,e_1,e_2) = (E,e'_1,e'_2)$ with $(\begin{smallmatrix} e'_1 \\ e'_2 \end{smallmatrix}) = \gamma (\begin{smallmatrix} e_1 \\ e_2 \end{smallmatrix})$ \cite[1.6]{Kat04}. We have an isomorphism of Riemann surfaces $(\Z/N\Z)^\times \times (\Gamma(N) \backslash \h) \xrightarrow{\cong} Y(N)(\C)$ given by $(\lambda, \tau) \mapsto \sabcd{1}{0}{0}{\lambda} \nu(\tau)$. For $N=2$, the moduli problem is not representable and we define $Y(2) = Y(2, 2)$ as a quotient of $Y(4)$ as in \cite[2.1]{Kat04}. It also has an action of $\GL_2(\Z/2\Z)$.

\begin{lemma} \label{lem OYN}
For any $N \geqslant 2$, the group $\mathcal{O}(Y(N))^\times$ can be identified with the group of modular units for $\Gamma(N)$ whose Fourier expansion at infinity has coefficients in $\Q(\zeta_N)$. Moreover, under this identification, the natural (right) action of $\GL_2(\Z/N\Z)$ on $\mathcal{O}(Y(N))^\times$ is described explicitly as follows:
\begin{enumerate}
\item For any $\gamma \in \SL_2(\Z/N\Z)$, we have $u | \gamma = u \circ \tilde{\gamma}$ for any representative $\tilde{\gamma}$ of $\gamma$ in $\SL_2(\Z)$.
\item For any $\lambda \in (\Z/N\Z)^\times$, we have $u | (\begin{smallmatrix} 1 & 0 \\ 0 & \lambda \end{smallmatrix}) = \sigma_\lambda(u)$, where $\sigma_\lambda(u)$ is obtained from $u$ by applying the automorphism $\zeta_N \mapsto \zeta_N^\lambda$ to the Fourier coefficients of $u$.
\end{enumerate}
\end{lemma}

Lemma \ref{lem OYN} is immediate if $Y(N)$ is defined using function fields, see \cite[Chapter~2, Section~2]{KL81} and \cite[Chapter 6]{Lan87}. For the equivalence between the two definitions of $Y(N)$, we refer to \cite{DR73}. For the convience of the reader, we sketch a direct proof of the lemma.

\begin{proof}[Sketch of proof of Lemma \ref{lem OYN}]
We use the fact that the Fourier expansion of a modular function at a cusp is obtained by ``evaluating at the Tate curve'' \cite[Section VII.3]{DR73}.
Consider the Tate curve $\mathbf{G}_m/q^\Z$ endowed with the level structure $(q^{1/N}, \zeta_N)$ (see \cite[VII]{DR73} and \cite[Section 8.8]{KM85}).
By the classical formulas for the parametrisation of the Tate curve \cite[Theorem V.1.1]{Sil94}, the triple $T = (\mathbf{G}_m/q^\Z, q^{1/N}, \zeta_N)$ is defined over the subring $R$ of $\Q(\zeta_N)((q^{1/N}))$ consisting of power series converging on the punctured unit disk $\{0 < |q^{1/N}| < 1\}$. Hence, for $N \geqslant 3$, the triple $T$ defines an element of $Y(N)(R)$.
The evaluation at a point $q = e^{2\pi i\tau}$ with $\tau \in \h$ gives a ring morphism $R \to \C$, and the image of $T$ in $Y(N)(\C)$ is $\nu(\tau)$.
Therefore, given $f \in \mathcal{O}(Y(N))$, the Fourier expansion of $f \circ \nu$ is $f(T)$ and lies in $\Q(\zeta_N)((q^{1/N}))$. By making explicit the action of $\GL_2(\Z/N\Z)$ on $T$ and the natural action of $\Aut(\C)$ on the power series $f(T)$, one shows that the map $f \mapsto f \circ \nu$ provides the desired identification.

For $N=2$, the affine ring $\mathcal{O}(Y(2))$ is obtained from $\mathcal{O}(Y(4))$ by taking invariants, and the result is deduced from the case $N=4$.
\end{proof}

For example, for any $\boldx \in (\Z/N\Z)^2$, $\boldx \neq \boldzero$, the function $g_\boldx^{12N}$ is a modular unit for $\Gamma(N)$ with Fourier coefficients in $\Q(\zeta_N)$. Lemma \ref{lem OYN} implies that $g_\boldx$ defines an element of $\mathcal{O}(Y(N))^\times \otimes \Z\bigl[\frac{1}{6N}\bigr]$. We call this element a \emph{Siegel unit} on $Y(N)$, in order to avoid confusion with the Siegel function. The action of $\GL_2(\Z/N\Z)$ on Siegel units is given by the relation $g_\boldx | \gamma = g_{\boldx \gamma}$ in $\mathcal{O}(Y(N))^\times \otimes \Z\bigl[\frac{1}{6N}\bigr]$. This follows from the transformation formula for the Siegel functions \cite[Proposition 5.1]{BN18} applied to the generators $\sabcd{1}{1}{0}{1}$ and $\sabcd{0}{-1}{1}{0}$ of $\SL_2(\Z)$.

\begin{proposition} \label{pro uabcd transform}
The modular unit $u(\bolda, \boldb, \boldc, \boldd)$ defines an element of $\mathcal{O}(Y(N))^\times$. Moreover, we have the following transformation formula:
\begin{equation*}
u(\bolda, \boldb, \boldc, \boldd) | \gamma = u(\bolda\gamma, \boldb\gamma, \boldc\gamma, \boldd\gamma) \qquad (\gamma \in \GL_2(\Z/N\Z)).
\end{equation*}
\end{proposition}

\begin{proof}
Equation \eqref{eq uxyzt} shows that the Fourier coefficients of $u(\bolda, \boldb, \boldc, \boldd)$ belong to $\Q(\zeta_N)$, so by Lemma \ref{lem OYN}, this modular unit defines an element of $\mathcal{O}(Y(N))^\times$. The transformation formula for $\gamma \in \SL_2(\Z/N\Z)$ follows from \eqref{wpx gamma}. The one for $\gamma = (\begin{smallmatrix} 1 & 0 \\ 0 & \lambda \end{smallmatrix})$ with $\lambda \in (\Z/N\Z)^\times$ follows from \eqref{eq uxyzt}, noting that the automorphism $\zeta_N \mapsto \zeta_N^\lambda$ transforms $\theta(\frac{1}{N}(x_1 \tau + x_2))$ into $\theta(\frac{1}{N}(x_1 \tau + \tilde{\lambda} x_2))$ for any $(x_1, x_2) \in \Z^2$, where $\tilde{\lambda}$ is a representative of $\lambda$ in $\Z$.
\end{proof}

We can specialise the units $u(\bolda, \boldb, \boldc, \boldd)$ to modular units for the group
\begin{equation*}
\Gamma_1(N) = \bigl\{ \gamma \in \SL_2(\Z) : \, \gamma \equiv \sabcd{1}{*}{0}{1} \mod N\bigr\}.
\end{equation*}

\begin{definition} \label{def u1abcd}
For any pairwise distinct elements $a,b,c,d$ of $(\Z/N\Z)/\pm 1$, we define
\begin{equation*}
u_1(a,b,c,d) = u((0,a), (0,b), (0,c), (0,d)).
\end{equation*}
\end{definition}

The transformation formula \eqref{wpx gamma} shows that $u_1(a,b,c,d)$ is a modular unit for $\Gamma_1(N)$. Proposition \ref{pro uabcd gx} also holds for these modular units. A closer look at the proof reveals that the root of unity $\zeta$ is $1$ in this case, so we get the cleaner identity between functions on $\h$
\begin{equation} \label{eq u1 g0x}
u_1(a, b, c, d) = \frac{g_{0, c+a} g_{0, c-a} g_{0, d+b} g_{0, d-b}}{g_{0, c+b} g_{0, c-b} g_{0, d+a} g_{0, d-a}}.
\end{equation}

The modular curve $Y_1(N)$ over $\Q$ is the quotient of $Y(N)$ by the subgroup $\{\sabcd{*}{*}{0}{1}\}$ of $\GL_2(\Z/N\Z)$. In the notation of \cite[2.1]{Kat04}, we have $Y_1(N) = Y(1,N)$, and $Y_1(N)$ agrees with the model of~\cite[Theorem 8.2.1]{DI95} for $N \geqslant 4$ since it represents the same moduli problem. The above description of the complex points of $Y(N)$ induces an isomorphism of Riemann surfaces $Y_1(N)(\C) \cong \Gamma_1(N) \backslash \h$.

\begin{proposition} \label{pro u1abcd}
The modular unit $u_1(a,b,c,d)$ defines an element of $\mathcal{O}(Y_1(N))^\times$.
\end{proposition}

\begin{proof}
This follows from Proposition \ref{pro uabcd transform}, noting that row vectors of the form $(0,x)$ with $x \in \Z/N\Z$ are right-invariant under the group $\{\sabcd{*}{*}{0}{1}\}$.
\end{proof}

\begin{example} \label{ex1}
We introduce our running example. Let $E$ be the closure of the affine curve $(1+x)^2 (1+y)^2 = xy$ in $\PP^1_\Q \times \PP^1_\Q$. Together with the rational point $(-1,\infty)$, this is an elliptic curve. One can show \cite{Bru23} that $E$ is isomorphic to the modular curve $X_1(15)$, the isomorphism sending $(-1,\infty)$ to the cusp $0$. Moreover
\begin{equation*}
x = -u_1(1,2,3,7) = - \zeta_{15} \cdot \frac{g_{0,2} g_{0,4}}{g_{0,1} g_{0,7}} \qquad
y = -u_1(2,4,6,1) = - \zeta_{15}^{-4} \cdot \frac{g_{0,4} g_{0,7}}{g_{0,1} g_{0,2}}.
\end{equation*}
These expressions in terms of Siegel functions follow from Proposition \ref{pro uabcd gx}. Moreover, the relation $u_1(1,2,3,7) + u_1(1,3,2,7) = 1$ yields the non-trivial identity
\begin{equation*}
\zeta_{15} \cdot g_{0,2} g_{0,4} g_{0,6} - \zeta_{15}^{-1} \cdot g_{0,1} g_{0,3} g_{0,4} = g_{0,1} g_{0,6} g_{0,7}.
\end{equation*}
Similarly, $u_1(2,4,6,1) + u_1(2,6,4,1)=1$ leads to an additive relation between products of three Siegel functions. Of course, Corollary \ref{cor additive relations} gives many more relations. We were not able to find a relation of the form
\begin{equation*}
\zeta \cdot g_{0,x} g_{0,y} + \zeta' \cdot g_{0,z} g_{0,t} = g_{0,u} g_{0,v}
\end{equation*}
at any level $N \leqslant 100$.
\end{example}

To conclude this section, we emphasise two properties of the modular units $u(\bolda, \boldb, \boldc, \boldd)$ that are crucial for our construction of $K_4$ elements in Section \ref{sec: construction}:

\begin{enumerate}
\item In Proposition \ref{pro uabcd gx}, the functions $u(\bolda, \boldb, \boldc, \boldd)$ are true modular units for $\Gamma(N)$, not just roots of modular units like the Siegel functions $g_\boldx$.
\item As witnessed by \eqref{eq uabcd uacbd}, the modular units $u = u(\bolda, \boldb, \boldc, \boldd)$ have the non-trivial property that $1-u$ is also a modular unit. Equivalently, they are solutions to the $S$-unit equation in the function field of $Y(N)$, where $S$ is the set of cusps (see Section \ref{subsec: de jeu curves}).
\end{enumerate}

\section{The $3$-term relations in $K_2$ of modular curves} \label{sec: 3-term}

We prove in this section $3$-term relations in the group $K_2$ of the modular curve $Y(N)$, together with a lift of these relations in Theorem \ref{triangulation manin3}.

We recall that there is a canonical antisymmetric bilinear map $\mathcal{O}(Y(N))^\times \times \mathcal{O}(Y(N))^\times \to K_2(Y(N))$ sending two modular units $u,v$ to the Milnor symbol $\{u,v\}$. As mentioned after Lemma \ref{lem OYN}, we may see the Siegel functions $g_\boldx$ as elements of $\mathcal{O}(Y(N))^\times \otimes \Z\bigl[\frac{1}{6N}\bigr]$, and we do so in the rest of this article.

\begin{theorem} \label{thm 3-term K2}
For any $\bolda, \boldb, \boldc \in (\Z/N\Z)^2$ such that $\bolda + \boldb + \boldc = \boldzero$, we have
\begin{equation} \label{eq 3-term K2}
\{g_\bolda, g_\boldb\} + \{g_\boldb, g_\boldc\} + \{g_\boldc, g_\bolda\} = 0 \qquad \textrm{in } K_2(Y(N)) \otimes \Z\Bigl[\frac{1}{6N}\Bigr].
\end{equation}
\end{theorem}

Theorem \ref{thm 3-term K2} can be seen as an analogue of the Manin $3$-term relations for modular symbols \cite{Man72}. Namely, fix a congruence subgroup $\Gamma$ of $\SL_2(\Z)$, and for any $h \in \SL_2(\Z)$, write $[h]$ for the hyperbolic geodesic from $h 0$ to $h \infty$ in the compactification $X_\Gamma = \Gamma \backslash (\h \cup \PP^1(\Q))$ of the modular curve $\Gamma \backslash \h$. Then, in the homology of $X_\Gamma$ relative to the cusps, we have the 3-term relation $[h] + [h t] + [h t^2] = 0$, where $t$ is the matrix $\sabcd{0}{-1}{1}{-1}$, which has order $3$. Now, for a matrix $M \in M_2(\Z/N\Z)$, define $\rho(M) = \{g_\bolda, g_\boldb\}$, where $\bolda$ and $\boldb$ are the two rows of~$M$. Then Theorem \ref{thm 3-term K2} is equivalent to $\rho(M) + \rho(tM) + \rho(t^2 M) = 0$ for any $M$.

Theorem \ref{thm 3-term K2} was previously known with $\Q$-coefficients \cite{Bru08, Gon08, BBPS23}. Recently, Sharifi and Venkatesh \cite{SV24} have constructed a Hecke equivariant map from the homology of $X_1(N)(\C)$ to $K_2(X_1(N)) \otimes \Z\bigl[\frac{1}{30N}\bigr]$, which gives a conceptual view on these $3$-term relations. This has connections to the group $K_2$ of the cyclotomic ring $\Z[\zeta_N]$ and to conjectures of Sharifi relating the modular curve $X_1(N)$ and the arithmetic of $\Q(\zeta_N)$ \cite{Bus08, FK24, Sha11}.

Our strategy to prove Theorem \ref{thm 3-term K2} builds on ideas in \cite{Bru08}. For each modular unit $u = u(\bolda, \boldb, \boldc, \boldd)$, the function $1-u = u(\bolda,\boldc,\boldb,\boldd)$ is also a modular unit, and we have the relation $\{u, 1-u\} = 0$ in $K_2(Y(N))$. Thanks to Proposition \ref{pro uabcd gx}, this generates plenty of linear dependence relations between the symbols $\{g_\boldx, g_\boldy\}$. We show that these relations are enough to prove \eqref{eq 3-term K2}. In fact, we can keep track of the relations $\{u, 1-u\} = 0$, thereby obtaining a lift of Theorem \ref{thm 3-term K2} in the exterior square of the group $\mathcal{O}(Y(N))^\times \otimes \Z\bigl[\frac{1}{6N}\bigr]$. To state this result, we introduce the following notation.

\begin{definition}
For any pairwise distinct elements $\bolda, \boldb, \boldc, \boldd$ of $(\Z/N\Z)^2/\pm 1$, define
\begin{equation*}
\delta(\bolda, \boldb, \boldc, \boldd) = u(\bolda, \boldb, \boldc, \boldd) \wedge u(\bolda, \boldc, \boldb, \boldd) \in \Lambda^2 \mathcal{O}(Y(N))^\times \otimes \Z\Bigl[\frac{1}{6N}\Bigr].
\end{equation*}
If $\bolda, \boldb, \boldc, \boldd$ are not pairwise distinct, we put $\delta(\bolda, \boldb, \boldc, \boldd)=0$.
\end{definition}

Theorem \ref{thm 3-term K2} is now a consequence of the following theorem, which is the main result of this section.

\begin{theorem} \label{triangulation manin3}
Let $G$ be a subgroup of $(\Z/N\Z)^2$, and let $\bolda, \boldb, \boldc \in G$ with $\bolda + \boldb + \boldc = \boldzero$. We have the following equality in $\Lambda^2 \mathcal{O}(Y(N))^\times \otimes \Z\bigl[\frac{1}{6N}\bigr]$:
\begin{equation} \label{eq manin3}
\begin{split}
g_\bolda \wedge g_\boldb + g_\boldb \wedge g_\boldc & + g_\boldc \wedge g_\bolda = \frac{1}{|G|} \sum_{\boldx \in G} \delta(\boldzero, \boldx, \bolda-\boldx, \boldb+\boldx) \\
& - \frac1{4|G|^2} \sum_{\boldx,\boldy \in G} \delta(\boldzero, \bolda, \boldc+2\boldx, \boldy) + \delta(\boldzero, \boldc, \boldb+2\boldx, \boldy) + \delta(\boldzero, \boldb, \bolda+2\boldx, \boldy).
\end{split}
\end{equation}
In the case $|G|$ is odd, this simplifies to
\begin{equation} \label{eq manin3 odd}
g_\bolda \wedge g_\boldb + g_\boldb \wedge g_\boldc + g_\boldc \wedge g_\bolda = \frac{1}{|G|} \sum_{\boldx \in G} \delta(\boldzero, \boldx, \bolda-\boldx, \boldb+\boldx).
\end{equation}
\end{theorem}

In analogy with modular symbols, where the chain $[h] + [h t] + [h t^2]$ is the boundary of an ideal hyperbolic triangle \cite{Man72}, we refer to \eqref{eq manin3} and \eqref{eq manin3 odd} as \emph{triangulations} of the 3-term relation in $K_2$.

\begin{remark}
We can work with the modular curve $Y_1(N)$ by choosing $G = \{0\} \times \Z/N\Z$ and using indices of the form $\boldx = (0,x)$. Indeed, the Siegel functions $g_{0,x}$ define elements of $\mathcal{O}(Y_1(N))^\times \otimes \Z\bigl[\frac{1}{6N}\bigr]$, and $u_1(a,b,c,d)$ is a modular unit on $Y_1(N)$ (Proposition \ref{pro u1abcd}).
\end{remark}

\begin{proof}[Proof of Theorem \ref{triangulation manin3}]
We will use the following two relations in $\mathcal{O}(Y(N))^\times \otimes \Z\bigl[\frac{1}{6N}\bigr]$:
\begin{equation} \label{eq uabcd gx 2}
u(\bolda, \boldb, \boldc, \boldd) = \frac{g_{\boldc+\bolda} g_{\boldc-\bolda} g_{\boldd+\boldb} g_{\boldd-\boldb}}{g_{\boldc+\boldb} g_{\boldc-\boldb} g_{\boldd+\bolda} g_{\boldd-\bolda}} \qquad \textrm{and} \qquad g_{-\boldx} = g_\boldx.
\end{equation}
The first one is a consequence of Proposition \ref{pro uabcd gx}, and the second one follows from \eqref{eq g-x}.

Let now $\bolda, \boldb, \boldc, \boldd$ be elements of $(\Z/N\Z)^2$ whose images in $(\Z/N\Z)^2/\pm 1$ are distinct. Expanding $\delta(\bolda, \boldb, \boldc, \boldd)$ using \eqref{eq uabcd gx 2}, we have
\begin{equation} \label{eq delta abcd}
\begin{split}
\delta(\bolda, \boldb, \boldc, \boldd) & = (g_{\boldc+\bolda} g_{\boldc-\bolda} \cdot g_{\boldd+\boldb} g_{\boldd-\boldb}) \wedge (g_{\boldb+\bolda} g_{\boldb-\bolda} \cdot g_{\boldd+\boldc} g_{\boldd-\boldc}) \\
& \;\quad + (g_{\boldb+\bolda} g_{\boldb-\bolda} \cdot g_{\boldd+\boldc} g_{\boldd-\boldc}) \wedge (g_{\boldc+\boldb} g_{\boldc-\boldb} \cdot g_{\boldd+\bolda} g_{\boldd-\bolda}) \\
& \;\quad + (g_{\boldc+\boldb} g_{\boldc-\boldb} \cdot g_{\boldd+\bolda} g_{\boldd-\bolda}) \wedge (g_{\boldc+\bolda} g_{\boldc-\bolda} \cdot g_{\boldd+\boldb} g_{\boldd-\boldb}).
\end{split}
\end{equation}
This expression is antisymmetric in $\bolda,\boldb,\boldc,\boldd$ and is zero when $\boldb=\pm \bolda$. Therefore \eqref{eq delta abcd} holds for any $\bolda, \boldb, \boldc, \boldd \in (\Z/N\Z)^2$. By expanding the right-hand side of \eqref{eq delta abcd} with respect to the dots, we get
\begin{equation} \label{delta phi}
\delta(\bolda, \boldb, \boldc, \boldd) = \varphi(\bolda, \boldb, \boldc) + \varphi(\boldc, \boldd, \bolda) + \varphi(\boldb, \bolda, \boldd) + \varphi(\boldd, \boldc, \boldb),
\end{equation}
where we have set
\begin{equation*}
\varphi(\boldx, \boldy, \boldz) = g_{\boldz+\boldx} g_{\boldz-\boldx} \wedge g_{\boldy+\boldx} g_{\boldy-\boldx} + g_{\boldy+\boldx} g_{\boldy-\boldx} \wedge g_{\boldz+\boldy} g_{\boldz-\boldy} + g_{\boldz+\boldy} g_{\boldz-\boldy} \wedge g_{\boldz+\boldx} g_{\boldz-\boldx}.
\end{equation*}
The functions $\delta$ and $\varphi$ are antisymmetric with respect to their arguments.

\begin{lemma} \label{lem sum phi}
For any $\boldy, \boldz \in G$, we have $\sum_{\boldx \in G} \varphi(\boldx, \boldy, \boldz) = 0$.
\end{lemma}

\begin{proof}
A simple computation shows that the sum simplifies to
\begin{equation*}
\sum_{\boldx \in G} \varphi(\boldx, \boldy, \boldz) = \sum_{\boldx \in G} g_{\boldz+\boldx} g_{\boldz-\boldx} \wedge g_{\boldy+\boldx} g_{\boldy-\boldx} = 2 \sum_{\boldx \in G} g_{\boldz+\boldx} \wedge g_{\boldy+\boldx} + g_{\boldz+\boldx} \wedge g_{\boldy-\boldx}.
\end{equation*}
Let us first consider $S = \sum_{\boldx \in G} g_{\boldz+\boldx} \wedge g_{\boldy+\boldx}$. Changing variables $\boldx = - \boldy - \boldz - \boldx'$, we get
\begin{equation*}
S = \sum_{\boldx' \in G} g_{-\boldy-\boldx'} \wedge g_{-\boldz-\boldx'} = \sum_{\boldx' \in G} g_{\boldy+\boldx'} \wedge g_{\boldz+\boldx'} = - S,
\end{equation*}
so that $S=0$. A similar argument using the change of variables $\boldx=\boldy-\boldz-\boldx'$ shows that the second part of the sum vanishes.
\end{proof}

\begin{lemma} \label{lem phi}
For any $\bolda, \boldb, \boldc \in G$, we have
\begin{equation*}
\varphi(\bolda, \boldb, \boldc) = \frac{1}{|G|} \sum_{\boldd \in G} \delta(\bolda, \boldb, \boldc, \boldd).
\end{equation*}
\end{lemma}

\begin{proof}
It follows from summing \eqref{delta phi} over $\boldd \in G$ and using Lemma \ref{lem sum phi}.
\end{proof}

Let $\psi(\bolda, \boldb) = g_\bolda \wedge g_\boldb  + g_\boldb \wedge g_\boldc + g_\boldc \wedge g_\bolda$, where $\boldc$ is defined by $\bolda + \boldb + \boldc = \boldzero$. Our next task is to show that $\psi(\bolda, \boldb)$ is a linear combination of values of $\varphi$. The definition of $\varphi$ gives
\begin{equation} \label{eq phi}
\varphi(\boldx, \boldy, \boldz) = \psi(\boldz+\boldx,-\boldy-\boldx) + \psi(\boldz+\boldx, \boldy-\boldx) + \psi(\boldz-\boldx, -\boldy+\boldx) + \psi(\boldz-\boldx, \boldy+\boldx).
\end{equation}
Changing variables and putting $\bolda=\boldz+\boldx$ and $\boldb=-\boldy-\boldx$, this becomes
\begin{align*}
\varphi(\boldx, -\boldb - \boldx, \bolda - \boldx) & = \psi(\bolda,\boldb) + \psi(\bolda, -\boldb-2\boldx) + \psi(\bolda-2\boldx, \boldb+2\boldx) + \psi(\bolda-2\boldx, -\boldb) \\
& = \psi(\bolda,\boldb) + \psi(-\bolda+\boldb+2\boldx, \bolda) + \psi(\boldb+2\boldx, \boldc) + \psi(\bolda-2\boldx, -\boldb).
\end{align*}
Here we used $\psi(\boldu, \boldv) = \psi(\boldv, - \boldu - \boldv) = \psi(- \boldu - \boldv, \boldu)$. Summing over $\boldx \in G$, we get
\begin{equation*}
\sum_{\boldx \in G} \varphi(\boldx, - \boldb - \boldx, \bolda - \boldx) = |G| \cdot \psi(\bolda, \boldb) + R_{- \bolda + \boldb}(\bolda) + R_\boldb(\boldc) + R_\bolda(-\boldb),
\end{equation*}
where $R_\boldu(\boldv) = \sum_{\boldx \in G} \psi(\boldu + 2\boldx, \boldv)$. One checks the relations $R_{\boldu + 2\boldw}(\boldv) = R_\boldu(\boldv)$ for any $\boldw \in G$, and $R_\boldu(-\boldv) = R_{-\boldu}(\boldv) = R_\boldu(\boldv)$. Therefore
\begin{equation} \label{sum phi psi R}
\sum_{\boldx \in G} \varphi(\boldx, -\boldb-\boldx, \bolda-\boldx) = |G| \cdot \psi(\bolda, \boldb) + R_\boldc(\bolda) + R_\boldb(\boldc) + R_\bolda(\boldb).
\end{equation}

\begin{lemma} \label{lem Ruv}
For any $\boldu, \boldv \in G$, we have $R_\boldu(\boldv) = \frac14 \sum_{\boldx \in G} \varphi(\boldzero, \boldv, \boldu + 2\boldx)$.
\end{lemma}

\begin{proof}
Taking $\boldx = \boldzero$ in \eqref{eq phi}, we obtain
\begin{equation*}
\varphi(\boldzero, \boldy, \boldz) = 2(\psi(\boldz, \boldy)+\psi(\boldz, -\boldy)) = 2(\psi(\boldz, \boldy)+\psi(-\boldz, \boldy)).
\end{equation*}
Specialising to $\boldy = \boldv$, $\boldz = \boldu + 2\boldx$, and summing over $\boldx \in G$ gives
\begin{equation*}
\sum_{\boldx \in G} \varphi(\boldzero, \boldv, \boldu + 2\boldx) = 2 \sum_{\boldx \in G} \psi(\boldu + 2\boldx, \boldv)+\psi(- \boldu - 2\boldx, \boldv) = 4 R_\boldu(\boldv). \qedhere
\end{equation*}
\end{proof}

Using Lemmas \ref{lem phi} and \ref{lem Ruv}, the equation \eqref{sum phi psi R} becomes
\begin{equation} \label{eq psi intermediate}
\begin{split}
\psi(\bolda, \boldb) & = \frac{1}{|G|^2} \sum_{\boldx, \boldy \in G} \delta(\boldx, - \boldb - \boldx, \bolda - \boldx, \boldy) \\
& \quad - \frac1{4|G|^2} \sum_{\boldx, \boldy \in G} \delta(\boldzero, \bolda, \boldc + 2\boldx, \boldy) + \delta(\boldzero, \boldc, \boldb + 2\boldx, \boldy) + \delta(\boldzero, \boldb, \bolda + 2\boldx, \boldy).
\end{split}
\end{equation}
We wish to simplify the first sum. For this, we will use the fact that $\delta$ satisfies $5$-term relations. More precisely, we have:

\begin{lemma} \label{lem delta 5-term}
Let $(\bolda_j)_{j \in \Z/5\Z}$ be a family of elements of $(\Z/N\Z)^2/\pm 1$. Then
\begin{equation} \label{eq delta 5-term}
\sum_{j \in \Z/5\Z} \delta(\bolda_j, \bolda_{j+1}, \bolda_{j+2}, \bolda_{j+3}) = 0 \quad \textrm{in } \Lambda^2 \mathcal{O}(Y(N))^\times \otimes \Z\Bigl[\frac{1}{6N}\Bigr].
\end{equation}
\end{lemma}

\begin{proof}
Assume first that the $\bolda_j$ are pairwise distinct. Let $F = \Q(Y(N))$ be the function field of $Y(N)$, and $F'$ be the field generated by the modular forms $\wp_\bolda$ with $\bolda \in (\Z/N\Z)^2 \setminus \{\boldzero\}$. Consider the formal sum
\begin{equation*}
D = \sum_{j \in \Z/5\Z} \{ u(\bolda_j, \bolda_{j+1}, \bolda_{j+2}, \bolda_{j+3}) \} \in \Z[F \setminus \{0,1\}],
\end{equation*}
where $\{x\}$ denotes the basis element of the free abelian group $\Z[F \setminus \{0,1\}]$. Let $h$ be the unique homography of $\PP^1(F')$ sending $\wp_{\bolda_0}, \wp_{\bolda_1}, \wp_{\bolda_2}$ to $\infty, 0, 1$, respectively. Let $x_j = h(\wp_{\bolda_j}) \in \PP^1(F)$. By projective invariance of the cross-ratio, we have
\begin{equation*}
D = \sum_{j \in \Z/5\Z} \{ [x_j, x_{j+1}, x_{j+2}, x_{j+3}] \}.
\end{equation*}
This is essentially a $5$-term relation in the sense of \cite[Section 1.2]{Gon95}. Taking into account Goncharov's different normalisation of the cross-ratio, we have in his notations
\begin{equation*}
D = - R_2(x_0, \ldots, x_4) + D'
\end{equation*}
where $D'$ is a linear combination of elements of the form $\{u\} + \{\frac{1}{u}\}$ and $\{u\} - \{1-\frac{1}{u}\}$ with $u \in F \setminus \{0,1\}$. We now apply the linear map
\begin{equation*}
\delta_2 : \Z[F \setminus \{0,1\}] \to \Lambda^2 F^\times, \quad \{x\} \mapsto (1-x) \wedge x.
\end{equation*}
(This is none other than the differential in the Bloch-Suslin complex of $F$, see Section \ref{subsec: goncharov complexes}.) We have $\delta_2(R_2(x_0, \ldots, x_4)) = 0$ by \cite[Section 1.8, p.~218]{Gon95}, and one checks directly that $\delta_2(D')$ is $2$-torsion. It follows that the left-hand side of \eqref{eq delta 5-term} is zero in $\Lambda^2 F^\times \otimes \Z\bigl[\frac{1}{6N}\bigr]$. But the divisor map identifies $F^\times / \mathcal{O}(Y(N))^\times$ with a subgroup of a free abelian group, hence $F^\times / \mathcal{O}(Y(N))^\times$ is free abelian. It follows that $\mathcal{O}(Y(N))^\times$ is a direct factor of $F^\times$. So the natural map $\Lambda^2 \mathcal{O}(Y(N))^\times \to \Lambda^2 F^\times$ is injective, showing \eqref{eq delta 5-term}.

In the case the $\bolda_j$ are not distinct, we may assume, by cyclic invariance, that $\bolda_0 = \bolda_1$ or $\bolda_0 = \bolda_2$. In the first case, the left-hand side of \eqref{eq delta 5-term} reduces to $\delta(\bolda_0, \bolda_2, \bolda_3, \bolda_4) + \delta(\bolda_2, \bolda_3, \bolda_4, \bolda_0)$. Putting $u = u(\bolda_0, \bolda_2, \bolda_3, \bolda_4)$, this equals
\begin{equation*}
u \wedge (1-u) + \frac{u}{u-1} \wedge \frac1{1-u},
\end{equation*}
which is $2$-torsion. The case $\bolda_0 = \bolda_2$ is similar.
\end{proof}

Lemma \ref{lem delta 5-term} gives in particular:
\begin{equation} \label{eq 5term}
\begin{split}
\delta(\boldx, - \boldb - \boldx, \bolda - \boldx, \boldy) & + \delta(- \boldb - \boldx, \bolda - \boldx, \boldy, \boldzero) + \delta(\bolda - \boldx, \boldy, \boldzero, \boldx) \\
& + \delta(\boldy, \boldzero, \boldx, - \boldb - \boldx) + \delta(\boldzero, \boldx, - \boldb - \boldx, \bolda - \boldx) = 0.
\end{split}
\end{equation}

\begin{lemma} \label{lem sum x delta}
For any $\boldsymbol\alpha, \boldsymbol\beta, \boldz, \boldt \in G$, we have $\sum_{\boldx \in G} \delta(\boldsymbol\alpha + \boldx, \boldsymbol\beta + \boldx, \boldz, \boldt) = 0$.
\end{lemma}

\begin{proof}
Denote this sum by $S$. The change of variables $\boldx = - \boldsymbol\alpha - \boldsymbol\beta - \boldx'$ gives
\begin{equation*}
S = \sum_{\boldx' \in G} \delta(- \boldsymbol\beta - \boldx', - \boldsymbol\alpha - \boldx', \boldz, \boldt) = \sum_{\boldx' \in G} \delta(\boldsymbol\beta + \boldx', \boldsymbol\alpha + \boldx', \boldz, \boldt) = -S. \qedhere
\end{equation*}
\end{proof}

Note that $\delta(\pm \bolda, \pm \boldb, \pm \boldc, \pm \boldd) = \delta(\bolda, \boldb, \boldc, \boldd)$. From \eqref{eq 5term} and Lemma \ref{lem sum x delta}, we obtain
\begin{equation*}
\sum_{\boldx \in G} \delta(\boldx, - \boldb - \boldx, \bolda - \boldx, \boldy) = - \sum_{\boldx \in G} \delta(\boldzero, \boldx, \boldb + \boldx, \bolda - \boldx) =  \sum_{\boldx \in G} \delta(\boldzero, \boldx, \bolda - \boldx, \boldb + \boldx).
\end{equation*}
Together with \eqref{eq psi intermediate}, this proves \eqref{eq manin3}. Finally, let us suppose that $|G|$ is odd. For any $\boldsymbol\alpha \in G$, the map $\boldx \mapsto \boldsymbol\alpha + 2\boldx$ is a bijection of $G$. Therefore, for any $\boldz, \boldt \in G$, we have
\begin{equation*}
\sum_{\boldx, \boldy \in G} \delta(\boldz, \boldt, \boldsymbol\alpha + 2\boldx, \boldy) = \sum_{\boldx, \boldy \in G} \delta(\boldz, \boldt, \boldx, \boldy) = 0
\end{equation*}
by antisymmetry with respect to $\boldx, \boldy$. Therefore the second line of \eqref{eq manin3} vanishes. This finishes the proof of Theorem \ref{triangulation manin3}.
\end{proof}

\begin{remark}
An alternative approach to Theorem \ref{triangulation manin3} would be to use Goncharov's strong Suslin reciprocity law \cite{Gon05, Rud21, Bol23}.
For $N \geqslant 3$, let $E(N)$ be the universal elliptic curve over the function field $k$ of $Y(N)$. Goncharov shows in \cite[Theorem 2.10]{Gon08} that the symbol $g_\bolda \wedge g_\boldb + g_\boldb \wedge g_\boldc + g_\boldc \wedge g_\bolda$ is the differential of a certain element $\theta_E(\bolda, \boldb, \boldc)$ in the Bloch-Suslin complex of $k$ (see Section \ref{subsec: goncharov complexes} for the definition of this complex). This gives another triangulation of the 3-term relation. The element $\theta_E(\bolda, \boldb, \boldc)$ is constructed using Goncharov's reciprocity law \cite[Theorem 6.14]{Gon05}. We note that Theorem \ref{triangulation manin3} gives a triangulation with modular units and not just with modular functions. It would be interesting to compare $\theta_E(\bolda, \boldb, \boldc)$ with the triangulations \eqref{eq manin3} and \eqref{eq manin3 odd}, using Goncharov's explicit formula in \cite[Theorem 6.14]{Gon05}.
\end{remark}

\begin{example} \label{ex2}
Continuing Example \ref{ex1}, consider the modular curve $X_1(15)$ and choose $\bolda = (0,1)$, $\boldb = (0,4)$. To simplify notation, define
\begin{equation*}
\delta_1(x,y,z,t) = \delta((0,x),(0,y),(0,z),(0,t))
\end{equation*}
for any $x,y,z,t$ in $(\Z/15\Z)/\pm 1$. Then Theorem \ref{triangulation manin3} gives $g_{0,1} \wedge g_{0,4} + g_{0,4} \wedge g_{0,-5} + g_{0,-5} \wedge g_{0,1}$ as a sum of $9$ terms of the form $\delta_1(x,y,z,t)$ (there are a priori 15 terms, but $6$ of them do not appear because the indices are not distinct in $(\Z/15\Z)/\pm 1$). Using a computer, we find the simpler triangulation
\begin{equation*}
g_{0,1} \wedge g_{0,4} + g_{0,4} \wedge g_{0,-5} + g_{0,-5} \wedge g_{0,1} = \frac15 \delta_1(0,1,2,4) + \frac15 \delta_1(0,1,3,6) - \frac25 \delta_1(1,2,3,5) - \frac15 \delta_1(1,2,4,6).
\end{equation*}
It can be checked directly using \eqref{eq delta abcd}.
\end{example}

\section{The polylogarithmic complex} \label{sec: goncharov de jeu}

In this section, we present the results of Goncharov and De Jeu that we need in order to construct elements in $K_4$ of modular curves.

For a scheme $X$, we denote by $H^i_\M(X, \Q(n))$ the weight $n$ motivic cohomology of $X$ with $\Q$-coefficients. If $X$ is smooth over a field, this group is isomorphic to the $n$th Adams eigenspace $K_{2n-i}^{(n)}(X)$ of Quillen's rational $K$-group $K_{2n-i}(X) \otimes \Q$.

\subsection{Goncharov's complexes} \label{subsec: goncharov complexes}

We begin by recalling Goncharov's theory of polylogarithmic complexes \cite{Gon95, Dup21}. For any abelian group $A$, we set $A_\Q = A \otimes_\Z \Q$. Let $F$ be a field, and let $n \geqslant 1$ be an integer. Goncharov constructs a \emph{weight $n$ polylogarithmic motivic complex} $\Gamma(F,n)$ of the following shape:
\begin{equation*}
B_n(F) \to B_{n-1}(F) \otimes F^\times_{\Q} \to B_{n-2}(F) \otimes \Lambda^2 F^\times_{\Q} \to \cdots \to B_2(F) \otimes \Lambda^{n-2} F^\times_{\Q} \to \Lambda^n F^\times_{\Q},
\end{equation*}
where $B_n(F)$ is defined as the quotient of the $\Q$-vector space $\Q[\PP^1(F)]$ with basis $\PP^1(F)$, by a certain subspace $R_n(F)$ reflecting the functional equations of the $n$th polylogarithm function. For $x \in \PP^1(F)$, we denote by $\{x\}_n$ the image of the basis element $\{x\}$ in $B_n(F)$. The differentials in the Goncharov complex are induced by maps $\delta_2 : B_2(F) \to \Lambda^2 F^\times_\Q$ and $\delta_n : B_n(F) \to B_{n-1}(F) \otimes F^\times_{\Q}$ for $n \geqslant 3$, which are uniquely determined by
\begin{equation*}
\delta_2(\{x\}_2) = \begin{cases} (1-x) \wedge x & \textrm{if } x \neq 0,1,\infty, \\
0 & \textrm{otherwise},
\end{cases}
\end{equation*}
and, for $n \geqslant 3$,
\begin{equation*}
\delta_n(\{x\}_n) = \begin{cases} \{x\}_{n-1} \otimes x & \textrm{if } x \neq 0,1,\infty, \\
0 & \textrm{otherwise}.
\end{cases}
\end{equation*}

The complex $\Gamma(F,n)$ sits in cohomological degrees $1$ to $n$ and is expected to compute the weight $n$ motivic cohomology of $\Spec F$. More precisely, combining \cite[Conjecture A and Conjecture 1.17, p. 222--223]{Gon95}, we have:

\begin{conjecture}[Goncharov] \label{Goncharov polylog conj}
For every $1 \leqslant i \leqslant n$, the group $H^i(\Gamma(F, n))$ is canonically isomorphic to $H^i_\M(F,\Q(n))$.
\end{conjecture}

For $i=n$, the group $H^n(\Gamma(F,n))$ is isomorphic to the Milnor $K$-group $K_n^M(F)_\Q$, and in this case the conjecture is known by \cite{NS89, Tot92}. For $n=2$, $i=1$, the conjecture is known by work of Suslin \cite{Sus90} (see the discussion below for more details). For further results, we refer the reader to the articles \cite{Gon95, Jeu95, Jeu96, Jeu00, GR25, Dup21}. Recently, Bolbachan \cite{Bol24} has announced a proof of the conjecture in the case $i=n-1$ for any $n \geqslant 2$ for fields of characteristic $0$.

In this article we will only need the polylogarithmic complexes of weight $2$ and $3$. We will take the version of these complexes where $R_n(F)$ is defined explicitly as in \cite[Section 1.8]{Gon95}, rather than inductively as in \cite[Section 1.9]{Gon95}. In particular, $B_2(F)$ is defined as follows.

\begin{definition}
The group $B_2(F)$ is the quotient of $\Q[F \setminus \{0,1\}]$ by the subspace $R_2(F)$ generated by the elements
\begin{equation*}
\{x\} - \{y\} + \{y/x\} - \Bigl\{\frac{1-x^{-1}}{1-y^{-1}}\Bigr\} + \Bigl\{\frac{1-x}{1-y}\Bigr\} \qquad (x,y \in F \setminus \{0,1\}, \; x \neq y).
\end{equation*}
\end{definition}

This definition agrees with Goncharov's definition in \cite[Section 1.8, p.~218]{Gon95} since every $5$-tuple $(x_0, \ldots, x_4)$ of distinct points of $\PP^1(F)$ is projectively equivalent to $(0, \infty, 1, x, y)$ for some $x,y$ in $F$.
The group $B_3(F)$ has a similar explicit definition as a quotient of $\Q[F \setminus \{0,1\}]$ \cite[Section 5.2]{Gon95}, but we will not need it in this article.

The complex $\Gamma(F,2) \colon B_2(F) \to \Lambda^2 F^\times_\Q$ (which is $B_F(2) \otimes \Q$ in the notation of \cite[Section~1.8, p.~218]{Gon95}) is the Bloch-Suslin complex of $F$ tensored with $\Q$. By Matsumoto's theorem \cite[Theorem III.6.1]{Wei13}, we have $H^2(\Gamma(F,2)) \cong K_2(F)_\Q$. The group $H^1(\Gamma(F,2))$ is called the Bloch group of $F$ (tensored with $\Q$). If $F$ is infinite, Suslin's theorem \cite[Theorem 5.2]{Sus90} implies that the Bloch group of $F$ is isomorphic to the quotient of $K_3(F)_\Q$ by the image of Milnor's $K$-group $K_3^M(F)_\Q$. Moreover, we have a decomposition $K_3(F)_\Q = K_3^{(2)}(F) \oplus K_3^{(3)}(F)$ by \cite[Section 2.8, Corollaire 1]{Sou85}, and $K_3^{(3)}(F)$ is the image of $K_3^M(F)_\Q$ in $K_3(F)_\Q$ \cite[Section 3.1, Théorème 2]{Sou85}. This implies that $H^1(\Gamma(F,2))$ is isomorphic to $K_3^{(2)}(F)$ if $F$ is infinite (see also \cite[Theorem VI.5.2]{Wei13}).

In weight $3$, the complex $\Gamma(F,3)$ (which is $B_F(3) \otimes \Q$ in the notation of \cite[Section~1.8, p.~219]{Gon95}) is placed in degrees 1 to 3 and looks as follows:
\begin{equation*}
\begin{tikzcd}[row sep = tiny]
\Gamma(F,3) : & B_3(F) \arrow{r} & B_2(F) \otimes F^\times_\Q \arrow{r} & \Lambda^3 F^\times_\Q \\
& \{x\}_3 \arrow[r, mapsto] & \{x\}_2 \otimes x & \\
& & \{x\}_2 \otimes y \arrow[r, mapsto] & (1-x) \wedge x \wedge y.
\end{tikzcd}
\end{equation*}
Goncharov's conjecture states that $H^2(\Gamma(F,3))$ is isomorphic to $K_4^{(3)}(F)$. In support of this, Goncharov constructed a canonical map $K_4(F)_\Q \to H^2(\Gamma(F,3))$ \cite[Section 6]{Gon95}.
This map should induce an isomorphism $K_4^{(3)}(F) \cong H^2(\Gamma(F,3))$ \cite[Conjectures 1.15 and 1.17]{Gon95}.

De Jeu \cite{Jeu95, Jeu96, Jeu00} constructed a map in the other direction by introducing other polylogarithmic complexes involving multi-relative $K$-theory. He proved the following theorem.

\begin{theorem}[De Jeu] \label{thm De Jeu map K4}
For any field $F$ of characteristic zero, there is a map
\begin{equation} \label{eq De Jeu map K4}
H^2(\Gamma(F,3)) \to K_4^{(3)}(F),
\end{equation}
which is canonical up to sign. Moreover, it is possible to choose the sign consistently for all fields so that the map \eqref{eq De Jeu map K4} becomes functorial in $F$.
\end{theorem}

This theorem follows from combining \cite[Section 2]{Jeu96}, where a map $H^2(\widetilde{\mathcal{M}}_{(3)}^{\bullet}(F)) \to K_4^{(3)}(F)$ is constructed (here $\widetilde{\mathcal{M}}_{(3)}^{\bullet}(F)$ is De Jeu's weight 3 polylogarithm complex), and \cite[Lemma 5.2]{Jeu00}, which can be used to define a map $H^2(\Gamma(F,3)) \to H^2(\widetilde{\mathcal{M}}_{(3)}^{\bullet}(F))$.

We will use Theorem \ref{thm De Jeu map K4} in the case $F$ is the function field of a modular curve. We note that Bolbachan \cite{Bol24} also constructed a map \eqref{eq De Jeu map K4} using higher Chow groups. Since we need to control the residues of the elements in $K_4^{(3)}(F)$, we will use De Jeu's map. We explain this map in more detail for curves in the next section.

\subsection{De Jeu's map for curves} \label{subsec: de jeu curves}

Let $X$ be a smooth projective geometrically connected curve defined over a number field $k$, and let $F=k(X)$ be its function field. In view of Theorem \ref{thm De Jeu map K4}, we search for degree 2 cocycles in the Goncharov complex $\Gamma(F,3)$. One obstacle is that the vector spaces appearing in the Goncharov complexes are infinite-dimensional. We now describe a way to circumvent this problem which works at least in certain cases.

Fix a finite set $S$ of closed points of $X$, and let $Y = X \backslash S$. By Soulé's localisation sequence for the $K$-theory with weights \cite[Remarques, p.~525]{Sou85}, we have an exact sequence
\begin{equation*}
\bigoplus_{x \in Y} K_4^{(2)}(k(x)) \to K_4^{(3)}(Y) \to K_4^{(3)}(F) \to \bigoplus_{x \in Y} K_3^{(2)}(k(x)),
\end{equation*}
where the map $K_4^{(3)}(F) \to K_3^{(2)}(k(x))$ is called the residue map at $x$. By Borel's theorem \cite[Theorem IV.1.18]{Wei13}, the groups $K_4^{(2)}(k(x))$ are zero. This gives an exact sequence
\begin{equation} \label{eq loc K4 2}
0 \to K_4^{(3)}(Y) \to K_4^{(3)}(F) \to \bigoplus_{x \in Y} K_3^{(2)}(k(x)),
\end{equation}
which enables us to view $K_4^{(3)}(Y)$ as a subspace of $K_4^{(3)}(F)$.

There are also residue maps at the level of the Goncharov complexes \cite[Sections 1.14--1.15]{Gon95}. More precisely, for any closed point $x \in X$, there is a morphism of complexes \cite[Section~1.15(c), p.~239]{Gon95}
\begin{equation*}
\Res_x \colon \; \Gamma(F,3) \to \Gamma(k(x),2)[-1]
\end{equation*}
which, in degree 2, sends the symbol $\{f\}_2 \otimes g$ to $\ord_x(g) \{f(x)\}_2$, with the convention $\{0\}_2 = \{1\}_2 = \{\infty\}_2 = 0$ in $B_2(k(x))$. Goncharov then defines $\Gamma(Y,3)$ as the simple complex (that is, the cone shifted by one) of the morphism of complexes
\begin{equation*}
\bigoplus_{x \in Y} \Res_x \colon \; \Gamma(F,3) \to \bigoplus_{x \in Y} \Gamma(k(x),2)[-1].
\end{equation*}
We thus have an exact sequence
\begin{equation*}
0 \to H^2(\Gamma(Y,3)) \to H^2(\Gamma(F,3)) \to \bigoplus_{x \in Y} H^1(\Gamma(k(x),2)),
\end{equation*}
which should be canonically isomorphic to \eqref{eq loc K4 2}.

In this direction, De Jeu has proved that the map $H^2(\Gamma(F,3)) \to K_4^{(3)}(F)$ commutes with taking residues up to an indeterminacy coming from $K_3^{(2)}(k)$. If $k$ is totally real, this group vanishes by Borel's theorem. This gives the following theorem.

\begin{theorem}[De Jeu] \label{thm DJ residues}
Assume $k$ is totally real. Then for every closed point $x \in X$, the following diagram commutes up to sign
\begin{equation*}
\begin{tikzcd}
H^2(\Gamma(F,3)) \arrow[d] \arrow[r, "2\Res_x"] & H^1(\Gamma(k(x), 2)) \arrow[d] \\
K_4^{(3)}(F) \arrow[r] & K_3^{(2)}(k(x)).
\end{tikzcd}
\end{equation*}
In particular De Jeu's map \eqref{eq De Jeu map K4} induces a map $H^2(\Gamma(Y,3)) \to K_4^{(3)}(Y)$.
\end{theorem}

Theorem \ref{thm DJ residues} follows from combining \cite[Corollary 5.4 and Remark 5.5]{Jeu96} and a direct check that the map $H^2(\Gamma(F,3)) \to H^2(\widetilde{\mathcal{M}}_{(3)}^{\bullet}(F))$ is compatible with taking residues, noting that \cite[Lemma 5.2]{Jeu00} allows one to define a map $H^1(\Gamma(k(x),2)) \to H^1(\widetilde{\mathcal{M}}_{(2)}^\bullet(k(x)))$.
 
For a general number field $k$, the following result will suffice for our needs.

\begin{theorem}[De Jeu] \label{construction K4Y}
Let $Y$ be a smooth (not necessarily projective) geometrically connected curve over a number field $k$, with function field $F=k(Y)$. Let $\xi = \sum_i n_i \{f_i\}_2 \otimes g_i$ be a degree 2 cocycle in the Goncharov complex $\Gamma(F,3)$, with $f_i, g_i \in F^\times$ and $n_i \in \Q$. Assume that all the functions $f_i$, $1-f_i$ and $g_i$ belong to $\mathcal{O}(Y)^\times$. Then the image of $\xi$ under De Jeu's map \eqref{eq De Jeu map K4} belongs to $K_4^{(3)}(Y)$.
\end{theorem}

\begin{proof}
The theorem follows essentially from \cite[Theorem 5.2]{Jeu96}. Let us indicate the details.
De Jeu builds the following complex $\mathcal{M}_{(3)}^\bullet(F)$ in degrees 1 to 3 (see \cite[Section 2]{Jeu00}):
\begin{equation*}
\begin{tikzcd}
M_{(3)}(F) \arrow{r} & M_{(2)}(F) \otimes F^\times_\Q \arrow[r, "\delta_2"] & F^\times_{\Q} \otimes \Lambda^2 F^\times_\Q,
\end{tikzcd}
\end{equation*}
where $M_{(n)}(F)$ is a $\Q$-vector space generated by elements $[x]_n$ with $x \in F \backslash \{0,1\}$.

Let $\xi' = \sum_i n_i [f_i]_2 \otimes g_i$ in $\mathcal{M}^2_{(3)}(F)$. Its boundary is $\delta_2(\xi') = \sum_i n_i \cdot (1-f_i) \otimes (f_i \wedge g_i)$ in $F^\times_\Q \otimes \Lambda^2 F^\times_\Q$. Let $W$ be the $\Q$-subspace of $F^\times_\Q$ generated by the functions $f_i$, $1-f_i$ and $g_i$. Since $\xi$ is a cocycle, $\sum_i n_i \cdot (1-f_i) \otimes f_i \otimes g_i$ maps to $0$ in $\Lambda^3 W$. This implies that $\delta_2(\xi')$ is a linear combination of symbols $u \otimes (u \wedge v)$ with $u,v \in W$. But such a symbol is the boundary of $([u]_2 + [u^{-1}]_2) \otimes v$, which maps to zero in the quotient complex $\widetilde{\mathcal{M}}^\bullet_{(3)}(F)$ constructed in \cite[Section 2, p.~529]{Jeu96}. We may thus modify $\xi'$ by such elements $([u]_2 + [u^{-1}]_2) \otimes v$ to get a cocycle $\xi''$ in $\mathcal{M}^2_{(3)}(F)$. Since $f_i,1-f_i,g_i$ are invertible on $Y$, we conclude by applying \cite[Theorem 5.2]{Jeu96} to $\xi''$ with $U=Y$.
\end{proof}

In view of Theorem \ref{construction K4Y}, we look for linear combinations of symbols $\{f\}_2 \otimes g$ with the additional condition that $f$, $1-f$ and $g$ are invertible on $Y$. We are thus led to consider the \emph{$S$-unit equation in $F$}
\begin{equation} \label{S-unit equation}
u + v = 1,
\end{equation}
where $u$ and $v$ are non-constant functions in $F$ whose zeros and poles are contained in $S$. It is known that this equation has only finitely many solutions (see Remark \ref{rk S-unit equation} below). Moreover $g$ lives in the $\Q$-vector space $\mathcal{O}(Y)^\times_\Q$, which is finite-dimensional after modding out by $k^\times_\Q$. Finally, the cocycle condition takes place in $\Lambda^3 \mathcal{O}(Y)^\times_\Q$, and the space $\Lambda^3 (\mathcal{O}(Y)^\times/k^\times)_\Q$ is also finite-dimensional. This essentially reduces our search to a linear algebra problem, which can be in principle implemented on a computer. Note that the $S$-unit equation may not have any solution if $S$ is too small. In this case, one may try to enlarge $S$, but it is not clear a priori which points should be added to $S$ in order to find solutions.

\begin{example} \label{ex3}
Continuing Example \ref{ex1}, let $E$ be the closure of the curve $(1+x)^2 (1+y)^2 = xy$ in $\PP^1_\Q \times \PP^1_\Q$. Let $Y = E \setminus S$ with $S = E(\Q)=\{(0,-1),(-1,0),(\infty,-1), (-1,\infty)\}$. Then the functions $f=-x$ and $g=-y$ satisfy the condition that $(f, 1-f, g, 1-g)$ belong to $\mathcal{O}(Y)^\times$. The boundary of the symbol $\tilde{\xi}_E = \{f\}_2 \otimes g - \{g\}_2 \otimes f$ is
\begin{equation*}
(1+x) \wedge (-x) \wedge (-y) - (1+y) \wedge (-y) \wedge (-x) = \frac12 \cdot (1+x)^2 (1+y)^2 \wedge x \wedge y = 0,
\end{equation*}
where we used the equation of $E$ (see also \cite[Section 4.1]{Lal15}). By Theorem \ref{construction K4Y}, the symbol $\tilde{\xi}_E$ gives rise to an element $\xi_E$ of $K_4^{(3)}(E \setminus S)$. Since $S$ consists of $\Q$-rational points and $K_3^{(2)}(\Q) = 0$, the residues of $\xi_E$ are trivial, so that $\xi_E$ actually belongs to $K_4^{(3)}(E)$.
\end{example}

\begin{remark} \label{rk S-unit equation}
The $S$-unit equation has mostly been studied when $F$ is a number field and $S$ is a set of finite places of $F$. It has applications to the study of integral points on algebraic curves, see the introduction of \cite[Chapter 8]{KL81}, \cite[Chapter 8]{Lan83} and \cite[Section IX.4]{Sil09}. The analogue of the identity \eqref{eq uabcd uacbd} for algebraic numbers is called Siegel's identity.

In the case $F$ is the function field of a smooth projective curve $X$ over an algebraically closed field, Mason \cite[Lemma, p.~222]{Mas83} has proved the following bound on the degree of a solution of \eqref{S-unit equation}:
\begin{equation} \label{eq mason bound}
\deg(u), \deg(v) \leqslant 2g-2+|S|,
\end{equation}
where $g$ is the genus of $X$ (see \cite{Sil84} for an alternative short proof; in the case $X = \PP^1$ over $\C$, this was first proved by Stothers \cite[Theorem 1.1]{Sto81}). As explained in \cite[p.~223]{Mas83}, the bound \eqref{eq mason bound} implies that the $S$-unit equation \eqref{S-unit equation} has only finitely many solutions, and in fact provides a method to find all the solutions. We implemented this method in Magma~\cite{Magma}, the idea being to loop over all possible divisors supported in $S$. In the case of elliptic curves, one may view this algorithm as an extension of Mellit's technique of parallel lines~\cite{Mel19}. Namely, the functions appearing in \cite{Mel19} have degree at most~$3$, while here the degree is arbitrary. Of course, looping over the divisors is unrealistic when the cardinality of $S$ or the Mason bound is too large; in practice, we are only able to deal with rather small degrees.

I learnt from A.~Javanpeykar that the finiteness of solutions can also be proved using the de Franchis-Severi theorem for hyperbolic curves \cite{Jav20}.
\end{remark}

\section{Constructing the elements in $K_4$ of modular curves} \label{sec: construction}

In this section, we construct the elements in $K_4^{(3)}(Y(N))$ using Theorem \ref{triangulation manin3} and the results of Goncharov and De Jeu from Section \ref{sec: goncharov de jeu}. We also explain in Section \ref{subsec: compactification} how to construct elements in $K_4^{(3)}(X(N))$, where $X(N)$ is the compactification of $Y(N)$. In Section~\ref{subsec: dep triangulation}, we study the impact of making other choices in the construction. In Section \ref{subsec: infinite level}, we define $K_4$ classes at infinite level. Finally, we show in Section \ref{subsec: modsym} that the elements constructed satisfy 3-term relations.

We view in this section $Y(N)$ as a curve over its field of constants $\Q(\zeta_N)$, so that $Y(N)$ is geometrically connected. We denote by $F$ be the function field of $Y(N)$.

In order to apply De Jeu's Theorem \ref{construction K4Y} with $Y = Y(N)$, we need to find symbols $\sum_i n_i \{f_i\}_2 \otimes g_i$ with $f_i \in F^\times$, $g_i \in F^\times_\Q$, $n_i \in \Q$, satisfying the following two conditions:

\begin{enumerate}
\item $\sum_i n_i \cdot f_i \wedge (1-f_i) \wedge g_i = 0$ in $\Lambda^3 F^\times_\Q$.
\item The functions $f_i$ and $1-f_i$ lie in $\mathcal{O}(Y(N))^\times$, and $g_i$ lies in $\mathcal{O}(Y(N))^\times_\Q$.
\end{enumerate}

As already mentioned at the end of Section \ref{sec: uabcd}, the modular units $u(\bolda, \boldb, \boldc, \boldd)$ fulfill the first part of condition (2). Note that these units are defined as cross-ratios, so it is very natural to consider the symbols $\{u(\bolda, \boldb, \boldc, \boldd)\}_2$ in $B_2(F)$.

\subsection{Definition of the elements} \label{def xiab xi1ab}

Let $\bolda, \boldb, \boldc \in (\Z/N\Z)^2$ with $\bolda + \boldb + \boldc = \boldzero$. We use Theorem \ref{triangulation manin3} with $G=(\Z/N\Z)^2$, and write the triangulation \eqref{eq manin3} as follows:
\begin{equation*}
g_\bolda \wedge g_\boldb + g_\boldb \wedge g_\boldc + g_\boldc \wedge g_\bolda = \sum_i n_i \cdot u_i \wedge (1-u_i) \qquad \textrm{in } \Lambda^2 F^\times_\Q,
\end{equation*}
with coefficients $n_i \in \Q$ and modular units $u_i \in \mathcal{O}(Y(N))^\times$.

\begin{definition} \label{def tilde xi ab}
Let $\tilde{\xi}(\bolda, \boldb)$ be the following cochain in the Goncharov complex of $F$:
\begin{equation} \label{eq tilde xi ab}
\tilde{\xi}(\bolda, \boldb) = \Bigl(\sum_i n_i \{u_i\}_2\Bigr) \otimes \frac{g_\boldb}{g_\bolda} \in \Gamma^2(F,3).
\end{equation}
\end{definition}

\begin{proposition} \label{pro xiab cocycle}
The cochain $\tilde{\xi}(\bolda, \boldb)$ is a cocycle, and its image in $K_4^{(3)}(F)$ under De Jeu's map belongs to $K_4^{(3)}(Y(N))$.
\end{proposition}

\begin{proof}
The boundary of $\tilde{\xi}(\bolda, \boldb)$ is given by
\begin{align*}
\delta(\tilde{\xi}(\bolda, \boldb)) & = \sum_i n_i \cdot (1-u_i) \wedge u_i \wedge \frac{g_\boldb}{g_\bolda} = - (g_\bolda \wedge g_\boldb + g_\boldb \wedge g_\boldc + g_\boldc \wedge g_\bolda) \wedge \frac{g_\boldb}{g_\bolda} \\
& = g_\boldb \wedge g_\boldc \wedge g_\bolda - g_\boldc \wedge g_\bolda \wedge g_\boldb = 0,
\end{align*}
so the condition (1) above is satisfied. Since the $u_i$ are of the form $u(\boldx, \boldy, \boldz, \boldt)$, the condition~(2) is also satisfied, and the result follows from De Jeu's Theorem \ref{construction K4Y}.
\end{proof}

\begin{definition} \label{def xi ab}
For any $\bolda, \boldb \in (\Z/N\Z)^2$, we denote by $\xi(\bolda, \boldb)$ the image of $\tilde{\xi}(\bolda, \boldb)$ in $K_4^{(3)}(Y(N))$.
\end{definition}

Assume we are in the special case $\bolda = (0,a)$ and $\boldb = (0,b)$ with $a,b \in \Z/N\Z$. Then we may use Theorem \ref{triangulation manin3} with the subgroup $G = \{0\} \times \Z/N\Z$. By Proposition \ref{pro u1abcd}, this produces a triangulation with modular units in $\mathcal{O}(Y_1(N))^\times$. We then define, as in \eqref{eq tilde xi ab},
\begin{equation} \label{eq xi1tilde}
\tilde{\xi}_1(a,b) = \Bigl(\sum_i n_i \{u_i\}_2\Bigr) \otimes \frac{g_{0,b}}{g_{0,a}},
\end{equation}
which now lives in the Goncharov complex of the function field of $Y_1(N)$. The same arguments as in Proposition \ref{pro xiab cocycle} show that $\tilde{\xi}_1(a,b)$ is a cocycle and that its image under De Jeu's map belongs to $K_4^{(3)}(Y_1(N))$.

\begin{definition} \label{def xi1 ab}
For any $a, b \in \Z/N\Z$, we denote by $\xi_1(a,b)$ the image of $\tilde{\xi}_1(a, b)$ in $K_4^{(3)}(Y_1(N))$.
\end{definition}

We will see in Sections \ref{sec: num beilinson} and \ref{sec: comparison} that the elements $\xi_1(a,b)$ can be non-trivial, by computing numerically their images under the Beilinson regulator map.

\begin{example} \label{ex4}
Continuing Example \ref{ex2}, consider the modular curve $Y_1(15)$. Then the element $\xi_1(1,4)$ is the image of the cocycle
\begin{equation*}
\tilde{\xi}_1(1,4) = \frac{1}{15} \Bigl(\sum_{x \in \Z/15\Z} \{u_1(0, x, 1-x, 4+x)\}_2\Bigr) \otimes \frac{g_{0,4}}{g_{0,1}},
\end{equation*}
where we only keep the terms for which $(0, x, 1-x, 4+x)$ are distinct in $(\Z/15\Z)/\pm 1$.
\end{example}

If $N$ is odd, then the triangulation \eqref{eq manin3 odd} leads to the same cocycles $\tilde{\xi}(\bolda, \boldb)$ and $\tilde{\xi}_1(a,b)$, and thus to the same elements $\xi(\bolda, \boldb)$ and $\xi_1(a,b)$. Indeed, for any $f \in F \setminus \{0,1\}$, we have $\{1/f\}_2 = - \{f\}_2$ in $B_2(F)$ by \cite[VI, Lemma 5.4(b)]{Wei13}. Then for fixed $\boldsymbol\alpha, \boldsymbol\beta \in (\Z/N\Z)^2$, we have in $B_2(F)$:
\begin{equation*}
\sum_{\boldx, \boldy \in (\Z/N\Z)^2} \{u(\boldsymbol\alpha, \boldsymbol\beta, \boldx, \boldy)\}_2 = \sum_{\boldx, \boldy \in (\Z/N\Z)^2} \{u(\boldsymbol\alpha,\boldsymbol\beta, \boldy, \boldx)^{-1}\}_2 = - \sum_{\boldx, \boldy \in (\Z/N\Z)^2} \{u(\boldsymbol\alpha, \boldsymbol\beta, \boldy, \boldx)\}_2,
\end{equation*}
which shows the claim. (In these sums, we only keep the terms for which $(\boldsymbol\alpha, \boldsymbol\beta, \boldx, \boldy)$ are distinct in $(\Z/N\Z)^2/\pm 1$.)

It is possible to produce $K_4$ elements for modular curves associated to arbitrary congruence subgroups, but in general the cocycles are not explicit anymore. More precisely, if $\Gamma$ is a subgroup of $\GL_2(\Z/N\Z)$ and $Y(\Gamma) = \Gamma \backslash Y(N)$ is the associated modular curve, we may consider the images of the elements $\xi(\bolda, \boldb)$ under the trace map $K_4^{(3)}(Y(N)) \to K_4^{(3)}(Y(\Gamma))$. However, consider the modular curve $Y_0(p)$ with $p$ prime. This curve has only two cusps, hence the group of modular units modulo the constants has rank $1$. In this case it is not possible to write down non-trivial cocycles using only modular units.

\subsection{Extension to the compactification} \label{subsec: compactification}

The elements $\xi(\bolda, \boldb)$ live on the open modular curve $Y(N)$. In view of Beilinson's conjecture on special values of $L$-functions, it is of interest to construct elements in the $K$-theory of a smooth projective variety. We explain in this section how to define elements in $K_4^{(3)}(X(N))$ and $K_4^{(3)}(X_1(N))$, where $X(N)$ and $X_1(N)$ are the smooth compactifications of $Y(N)$ and $Y_1(N)$, respectively.
We will use without mention Borel's theorem on the $K$-groups of number fields \cite[Theorem IV.1.18]{Wei13}.

Let $S = X(N) \setminus Y(N)$ be the closed subscheme of cusps of $X(N)$. We have the localisation exact sequence in $K$-theory
\begin{equation} \label{eq loc XY}
0 \to K_4^{(3)}(X(N)) \to K_4^{(3)}(Y(N)) \xrightarrow{\Res} K_3^{(2)}(S)
\end{equation}
(see \cite[Théorème 9]{Sou85}, \cite[Corollary, p.~167]{Tam88}). We say that an element $\xi$ of $K_4^{(3)}(Y(N))$ extends to $X(N)$ if it is the image of a (unique) element of $K_4^{(3)}(X(N))$ in \eqref{eq loc XY}. This amounts to say that the residue of $\xi$ at every cusp is trivial. Although the elements $\xi(\bolda, \boldb)$ do not a priori extend to $X(N)$, we can modify them so that they do extend.

\begin{proposition} \label{pro localisation retraction}
The restriction map $K_4^{(3)}(X(N)) \to K_4^{(3)}(Y(N))$ admits a canonical retraction.
\end{proposition}

This proposition is an analogue of Bloch's trick to construct an element of $K_2^{(2)}(X)$ from an element of $K_2^{(2)}(X \setminus S)$, where $X$ is a smooth projective curve and $S$ is a finite subset of~$X$ having the property that the difference of any two geometric points of $S$ has finite order in the Jacobian of $X$ (see \cite[(8.2)]{Blo10} for the case of elliptic curves).

\begin{proof}
Let $k = \Q(\zeta_N)$. Recall that the cusps of $X(N)$ are all defined over $k$. Write $i \colon S \hookrightarrow X(N)$ for the closed immersion and $\pi \colon X(N) \to \Spec k$ for the structural morphism. We have the following diagram
\begin{equation} \label{eq diag K4}
\begin{tikzcd}[sep=scriptsize]
0 \arrow{r} & K_4^{(3)}(X(N)) \arrow{r} & K_4^{(3)}(Y(N)) \arrow{r}{\Res} & \displaystyle \bigoplus_{x \in S} K_3^{(2)}(k) \arrow{r}{i_*} \arrow[dr, swap, "\Sigma"] & K_3^{(3)}(X(N)) \arrow{d}{\pi_*} \\
& & & &  K_3^{(2)}(k)
\end{tikzcd}
\end{equation}
where the first row is the localisation exact sequence. The diagonal arrow $\Sigma$ is the sum map because $\pi \circ i \colon S \to \Spec k$ consists of copies of the identity map of $\Spec k$.

Let $T$ be the subspace of $K_4^{(3)}(Y(N))$ generated by the elements of the form $\lambda \cup u$ where $\lambda \in K_3^{(2)}(k)$ and $u \in \mathcal{O}(Y(N))^\times$. The residue of $\lambda \cup u$ in the above sequence is given by $\lambda \otimes \dv(u)$ (see \cite[1.3.2.(2)]{Deg08} applied to the closed immersion $i$ and to $Y = \Spec k$).  By the Manin--Drinfel'd theorem \cite{Dri73}, the difference of any two cusps of $X(N)$ is torsion in the Jacobian of $X(N)$. This implies that $\Res(T) = \ker(\Sigma)$. We claim that
\begin{equation} \label{eq decomposition H23Y}
K_4^{(3)}(Y(N)) = K_4^{(3)}(X(N)) \oplus T.
\end{equation}
The fact that $K_4^{(3)}(Y(N))$ is generated by $K_4^{(3)}(X(N))$ and $T$ follows from the diagram \eqref{eq diag K4}. Now consider the composite map
\begin{equation*}
K_3^{(2)}(k) \otimes \mathcal{O}(Y(N))^\times \stackrel{\cup}{\longrightarrow} T \xrightarrow{\Res} \bigoplus_{x \in S} K_3^{(2)}(k).
\end{equation*}
The kernel of this map is $K_3^{(2)}(k) \otimes k^\times$. Therefore the intersection of $K_4^{(3)}(X(N))$ and $T$ is contained in $K_4^{(3)}(k) = 0$. The decomposition \eqref{eq decomposition H23Y} provides the desired retraction.
\end{proof}

\begin{definition}
For $\bolda, \boldb \in (\Z/N\Z)^2$, we denote by $\xi'(\bolda, \boldb) \in K_4^{(3)}(X(N))$ the image of $\xi(\bolda, \boldb)$ under the retraction of Proposition \ref{pro localisation retraction}.
\end{definition}

We can proceed similarly with the modular curve $Y_1(N)$. As in \eqref{eq loc XY}, there is a localisation sequence
\begin{equation*}
0 \to K_4^{(3)}(X_1(N)) \to K_4^{(3)}(Y_1(N)) \xrightarrow{\Res} K_3^{(2)}(S_1),
\end{equation*}
where $S_1$ is the subscheme of cusps of $X_1(N)$. The cusps of $X_1(N)$ are defined over $\Q(\zeta_N)$, and the same proof as in Proposition \ref{pro localisation retraction} shows that the restriction map
\begin{equation*}
K_4^{(3)}(X_1(N)_{\Q(\zeta_N)}) \to K_4^{(3)}(Y_1(N)_{\Q(\zeta_N)})
\end{equation*}
has a $\Gal(\Q(\zeta_N)/\Q)$-equivariant retraction. Since motivic cohomology with $\Q$-coefficients satisfies Galois descent (see \cite[Section C.1, Theorem 5]{CD19} or \cite[1.3(6)]{DS91}), this provides a canonical retraction
\begin{equation} \label{eq retraction X1N}
K_4^{(3)}(Y_1(N)) \to K_4^{(3)}(X_1(N)).
\end{equation}

\begin{definition}
For $a, b \in \Z/N\Z$, we denote by $\xi'_1(a, b) \in K_4^{(3)}(X_1(N))$ the image of $\xi_1(a, b)$ under the retraction \eqref{eq retraction X1N}.
\end{definition}

In order to give an example, we first recall some facts about the cusps of $X_1(N)$, and the Galois action on them. The set $\Gamma_1(N) \backslash \PP^1(\Q)$ of cusps of $X_1(N)(\C)$ is in bijection with
\begin{equation*}
\{(c,d) : c \in \Z/N\Z, \, d \in (\Z/(c,N)\Z)^\times\}/ \pm 1.
\end{equation*}
This bijection sends a cusp $\gamma \infty$ with $\gamma \in \SL_2(\Z)$ to the class of the bottom row of $\gamma$ (see \cite[Example 9.1.3]{DI95}). The action of $\Aut(\C)$ on the set of cusps is then given as follows (compare with \cite[Theorem 1.3.1]{Ste82}, where another model of $X_1(N)$ is used).

\begin{lemma} \label{lem galois action cusps}
Let $\sigma \in \Aut(\C)$. If a cusp $x$ of $X_1(N)(\C)$ is represented by the pair $(c,d)$ using the above bijection, then $\sigma(x)$ is represented by $(c, \chi(\sigma) d)$, where $\chi(\sigma) \in (\Z/N\Z)^\times$ is defined by $\sigma(e^{2\pi i/N}) = e^{2\pi i \chi(\sigma)/N}$.
\end{lemma}

\begin{proof}
(See \cite[Section 9.3, p.~79]{DI95}.) Using the interpretation of $X_1(N)$ as a moduli space of generalised elliptic curves, the cusp $x$ corresponds to the isomorphism class of a Néron $N'$-gon $E$ with $N' = N/(c,N)$, equipped with the point $P = e^{2\pi id/N}$ on the $c'$th component of $E$, with $c' = c/(c,N) \in \Z/N'\Z$. It follows that $\sigma(x)$ corresponds to the same Néron $N'$-gon with the point $\sigma(P) = e^{2\pi i \chi(\sigma) d/N}$ on the $c'$th component.
\end{proof}

\begin{proposition} \label{pro N prime}
Assume $N=p$ or $N=2p$ where $p$ is prime. Then for any $a,b$ in $\Z/N\Z$, the element $\xi_1(a,b)$ extends to $X_1(N)$, and we have $\xi'_1(a,b)=\xi_1(a,b)$.
\end{proposition}

\begin{proof}
Using Lemma \ref{lem galois action cusps} with the complex conjugation, we see that all the cusps of $X_1(N)(\C)$ are real. It follows that the residue field $\Q(x)$ of a cusp $x$ of $X_1(N)$ is totally real, and this implies that $K_3^{(2)}(\Q(x)) = 0$. Since this is true for every $x$, the elements $\xi_1(a,b)$ extend to $X_1(N)$, and $\xi'_1(a,b) = \xi_1(a,b)$.
\end{proof}

We now turn to our running example. We will need the following general result to compute the divisor of a Siegel unit on $Y_1(N)$.

\begin{lemma} \label{lem div g0a}
Let $N \geqslant 5$ be an integer. Let $a \in \Z/N\Z$, $a \neq 0$, and let $x$ be a cusp of $X_1(N)(\C)$, represented by the pair $(c,d)$ as above. Then the order of vanishing of $g_{0,a}$ at the cusp $x$ is given by
\begin{equation*}
\operatorname{ord}_x(g_{0,a}) = \frac{N}{2(c,N)} B_2\Bigl(\frac{\widetilde{ac}}{N}\Bigr),
\end{equation*}
where $\widetilde{ac}$ is the representative of $ac$ in $\{0, \ldots, N-1\}$.
\end{lemma}

\begin{proof}
Let $\gamma \in \SL_2(\Z)$ such that $\gamma \infty = x$ and the bottom row of $\gamma$ is congruent to $(c,d)$. By the transformation formula for Siegel functions under $\SL_2(\Z)$ \cite[Section 5.1]{BN18}, we have $g_{0,a} \circ \gamma = \zeta \cdot g_{ac, ad}$, where $\zeta$ is a root of unity. By Definition \ref{def gx}, the Fourier expansion of $g_{ac,ad}$ at infinity is of the form $\alpha q^{\lambda} + \cdots$, where $\alpha$ is a non-zero constant and $\lambda = \frac12 B_2\bigl(\frac{\widetilde{ac}}{N}\bigr)$. A local coordinate at the cusp $x$ of $X_1(N)(\C)$ is given by $e^{2\pi i (\gamma^{-1}\tau)/w}$, where $w = N/(c,N)$ is the width of $x$ (see \cite[Section 9.1]{DI95} and \cite[Section 6.3]{CS17}). The formula follows.
\end{proof}

\begin{example} \label{ex5}
Consider the element $\xi_1(1,4)$ of $K_4^{(3)}(Y_1(15))$ from Example \ref{ex4}. Since the modular curve $X_1(15)$ is defined over $\Q$, we can use De Jeu's Theorem \ref{thm DJ residues} to compute the residues of $\xi_1(1,4)$ at the cusps. Looking at the definition of the residue map on the Goncharov complex, we have to compute the divisor of $g_{0,4}/g_{0,1}$. With Lemma \ref{lem div g0a}, we find
\begin{equation*}
\dv\Bigl(\frac{g_{0,4}}{g_{0,1}}\Bigr) = - [0] - \Bigl[\frac12\Bigr] + \Bigl[\frac14\Bigr] + \Bigl[\frac17\Bigr].
\end{equation*}
By Lemma \ref{lem galois action cusps}, the cusps involved in this divisor are defined over $\Q$. Since $K_3^{(2)}(\Q) = 0$, we conclude that $\xi_1(1,4)$ extends to $X_1(15)$.
\end{example}

The numerical computations in Section \ref{sec: num beilinson} show that the elements $\xi_1(a,b)$ do not always extend to $X_1(N)$. For example, in the case $N=15$, the element $\xi_1(1,3)$ has a non-trivial residue at the cusp $1/5$.

\subsection{Dependence on the triangulation} \label{subsec: dep triangulation}

We defined the various elements $\xi(\bolda, \boldb)$, $\xi_1(a,b)$, $\xi'(\bolda, \boldb)$, $\xi'_1(a,b)$ using the triangulation of $g_\bolda \wedge g_\boldb + g_\boldb \wedge g_\boldc + g_\boldc \wedge g_\bolda$ provided by Theorem \ref{triangulation manin3}, but we could have chosen any other triangulation to define elements in $K_4$. In this section, we study the dependence on this choice.

Say we have two triangulations in $\Lambda^2 F^\times_\Q$:
\begin{equation*}
g_\bolda \wedge g_\boldb + g_\boldb \wedge g_\boldc + g_\boldc \wedge g_\bolda = \sum_i n_i \cdot u_i \wedge (1-u_i) = \sum_j n'_j \cdot v_j \wedge (1-v_j).
\end{equation*}
Then
\begin{equation*}
\sum_i n_i \{u_i\}_2 - \sum_j n'_j \{v_j\}_2
\end{equation*}
is an element of $H^1(\Gamma(F,2))$. Suslin's rigidity conjecture \cite[Conjecture 5.4]{Sus86} asserts that the canonical map $H^1(\Gamma(\Q(\zeta_N), 2)) \to H^1(\Gamma(F,2))$ is an isomorphim. Therefore the two triangulations differ by an element $\lambda$ of $H^1(\Gamma(\Q(\zeta_N), 2)) \cong K_3^{(2)}(\Q(\zeta_N))$. It should be the case that the two resulting elements in $K_4^{(3)}(Y(N))$ differ by $\lambda \cup (g_\boldb/g_\bolda)$. This would follow from the compatibility of De Jeu's map \eqref{eq De Jeu map K4} with the cup-product $K_3^{(2)} \times K_1^{(1)} \to K_4^{(3)}$. As in the proof of Proposition \ref{pro localisation retraction}, the residue of $\lambda \cup (g_\boldb/g_\bolda)$ at a cusp $x$ is equal to $\ord_x(g_\boldb/g_\bolda) \lambda$, where $\ord_x$ denotes the order of vanishing at $x$. Since $\Q(\zeta_N)$ is not totally real for $N \geqslant 3$, this shows that $\xi(\bolda, \boldb)$ can depend on the triangulation. However, the element $\lambda \cup (g_\boldb/g_\bolda)$ is killed by the retraction of Proposition \ref{pro localisation retraction}. Therefore, the element $\xi'(\bolda, \boldb)$ of $K_4^{(3)}(X(N))$ should not depend on the triangulation.

For the modular curve $X_1(N)$, since its field of constants is $\Q$ and $K_3^{(2)}(\Q) = 0$, we have the following conditional independence result.

\begin{proposition} \label{pro indep triangulation}
Let $a, b \in \Z/N\Z$. Assuming Suslin's rigidity conjecture for the function field of $X_1(N)$, the elements $\xi_1(a,b)$ and $\xi'_1(a,b)$ do not depend on the choice of triangulation.
\end{proposition}

It seems hard to prove the independence for arbitrary triangulations. However, the following problem may be more accessible. The proof of Lemma \ref{lem delta 5-term} shows that for any family $(\bolda_j)_{j \in \Z/5\Z}$ of elements of $(\Z/N\Z)^2/\pm 1$, we have the $5$-term relation in $B_2(F)$
\begin{equation} \label{eq 5term B2}
\sum_{j \in \Z/5\Z} \{u(\bolda_j, \bolda_{j+1}, \bolda_{j+2}, \bolda_{j+3})\}_2 = 0.
\end{equation}
There are also ``unexpected'' duplicates among the $u(\bolda, \boldb, \boldc, \boldd)$, in other words equalities
\begin{equation} \label{eq u equal}
u(\bolda, \boldb, \boldc, \boldd) = u(\bolda', \boldb', \boldc', \boldd')
\end{equation}
which do not follow from the definition as a cross-ratio. In order to prove \eqref{eq u equal}, one computes the divisors of both sides using Proposition \ref{pro uabcd gx} and a straightforward generalisation of Lemma~\ref{lem div g0a}, together with the leading coefficient of the Fourier expansion of these modular units at infinity using Proposition \ref{pro uabcd gx}. In the case of the modular units $u_1(a,b,c,d)$, it is in fact simpler to consider the leading coefficient at the cusp $0$, as one can show using the transformation formula stated in the proof of Lemma \ref{lem div g0a} that this leading coefficient is always $\pm1$.

This raises the following question: given two triangulations of $g_\bolda \wedge g_\boldb + g_\boldb \wedge g_\boldc + g_\boldc \wedge g_\bolda$ involving only units of the form $u(\boldx, \boldy, \boldz, \boldt)$, can one show, using \eqref{eq 5term B2} and \eqref{eq u equal}, that they give rise to the same element of $B_2(F)$? At least, can one show that the two resulting elements in $K_4^{(3)}(Y(N))$ are equal, using also coboundaries in the Goncharov complex $\Gamma(F, 3)$?

\begin{example} \label{ex6}
Consider again the element $\xi_1(1,4)$ in $K_4^{(3)}(Y_1(15))$. We found in Example~\ref{ex2} another triangulation for $g_{0,1} \wedge g_{0,4} + g_{0,4} \wedge g_{0,-5} + g_{0,-5} \wedge g_{0,1}$. By Proposition \ref{pro indep triangulation}, it follows that under Suslin's conjecture, $\xi_1(1,4)$ is also the image of the cocycle
\begin{equation*}
\frac15 \bigl(\{u_1(0,1,2,4)\}_2 + \{u_1(0,1,3,6)\}_2 - 2 \{u_1(1,2,3,5)\}_2 - \{u_1(1,2,4,6)\}_2 \bigr) \otimes \frac{g_{0,4}}{g_{0,1}}.
\end{equation*}
Using \eqref{eq 5term B2}, \eqref{eq u equal} and a computer, one can prove that the two aforementioned triangulations give rise to the same element in $B_2(\Q(Y_1(15)))$. Therefore, the two resulting elements of $K_4^{(3)}(Y_1(15))$ are equal unconditionally.
\end{example}

\subsection{Classes at infinite level} \label{subsec: infinite level}

We now investigate how the elements $\xi(\bolda, \boldb)$ and $\xi'(\bolda, \boldb)$ vary with $N$. It will be convenient to work with the tower of modular curves $Y(N)$, using the canonical projection maps $Y(N') \to Y(N)$ for $N$ dividing $N'$.

Given $\boldx \in (\Q/\Z)^2$, choose $N \geqslant 1$ such that $N\boldx = \boldzero$. Identifying $\frac{1}{N} \Z/\Z$ with $\Z/N\Z$, we have a Siegel unit $g_\boldx$ in $\mathcal{O}(Y(N))^\times \otimes \Q$. From Definition \ref{def gx}, we see that its image in
\begin{equation*}
\mathcal{O}(Y(\infty))^\times_\Q := \varinjlim_{N \geqslant 1} \mathcal{O}(Y(N))^\times_\Q
\end{equation*}
does not depend on the choice of $N$.

Now, let $\bolda, \boldb$ be two elements of $(\Q/\Z)^2$, and choose an integer $N \geqslant 1$ such that $N \bolda = N \boldb = \boldzero$. Proceeding similarly, we obtain an element
\begin{equation*}
\xi'(\bolda, \boldb) \in K_4^{(3)}(X(\infty)) := \varinjlim_{N \geqslant 1} K_4^{(3)}(X(N)).
\end{equation*}
Since a triangulation in level $N$ gives a triangulation in any level $N'$ divisible by $N$, the discussion in Section \ref{subsec: dep triangulation} shows that this $\xi'(\bolda, \boldb)$ should not depend on the choices of $N$ and of the triangulation. If we stick to triangulations involving units of the form $u(\boldx, \boldy, \boldz, \boldt)$, one way to show this would be to use \eqref{eq 5term B2} and \eqref{eq u equal}. This may even give a well-defined class $\xi(\bolda, \boldb)$ on $Y(\infty)$.

In the case of the tower of modular curves $X_1(N)$, we may define similarly, for $a,b \in \Q/\Z$ and for a choice of $N$ such that $Na = Nb = 0$,
\begin{equation*}
\begin{split}
\xi_1(a,b) \in K_4^{(3)}(Y_1(\infty)) & := \varinjlim_{N \geqslant 1} K_4^{(3)}(Y_1(N)), \\
\xi'_1(a,b) \in K_4^{(3)}(X_1(\infty)) & := \varinjlim_{N \geqslant 1} K_4^{(3)}(X_1(N)).
\end{split}
\end{equation*}
Under Suslin's conjecture, these elements are independent of choices by Proposition \ref{pro indep triangulation}, and one may try to show this unconditionally for triangulations involving only the units $u_1(x,y,z,t)$.

\subsection{Analogy with modular symbols} \label{subsec: modsym}

To conclude this section, we show that the elements $\xi(\bolda, \boldb)$ satisfy $3$-term relations, like the Beilinson--Kato elements $\{g_\bolda, g_\boldb\}$ in the group $K_2^{(2)}(Y(N))$. We actually expect $3$-term relations for the Beilinson elements in the group $K_{2n-2}^{(n)}(Y(N)) \cong H^2_\M(Y(N), \Q(n))$ for any $n \geqslant 2$, but this is not known for $n \geqslant 3$. (For the comparison between $\xi(\bolda, \boldb)$ and the Beilinson elements, see Section \ref{sec: comparison}.)

First, the elements $\xi(\bolda, \boldb)$ satisfy the transformation formula
\begin{equation*}
\xi(\bolda, \boldb) | \gamma = \xi(\bolda\gamma, \boldb\gamma) \qquad (\bolda, \boldb \in (\Z/N\Z)^2, \; \gamma \in \GL_2(\Z/N\Z)).
\end{equation*}
This is actually true at the level of the cocycles $\tilde{\xi}(\bolda, \boldb)$, thanks to Proposition \ref{pro uabcd transform} and the transformation formula $g_\bolda | \gamma = g_{\bolda\gamma}$ mentioned before that proposition.

\begin{proposition} \label{pro xiG 3-term}
For any $\bolda, \boldb, \boldc \in (\Z/N\Z)^2$ with $\bolda + \boldb + \boldc = \boldzero$, we have
\begin{equation} \label{eq xiG 3-term}
\xi(\bolda, \boldb) + \xi(\boldb, \boldc) + \xi(\boldc, \bolda) = 0,
\end{equation}
and similarly for the elements $\xi'(\bolda, \boldb)$. If $N$ is odd, we also have $\xi(\boldb, \bolda) = \xi(\bolda, \boldb)$ for any $\bolda, \boldb \in (\Z/N\Z)^2$.
\end{proposition}

\begin{proof}
We show that \eqref{eq xiG 3-term} already holds at the level of cocycles. Write
\begin{equation} \label{eq Tab}
\begin{split}
T(\bolda, \boldb) = & \frac{1}{N^2} \sum_{x \in (\Z/N\Z)^2} \{u(\boldzero, \boldx, \bolda-\boldx, \boldb+\boldx)\}_2 \\
& - \frac1{4N^4} \sum_{\boldx, \boldy \in (\Z/N\Z)^2} \{u(\boldzero, \bolda, \boldc+2\boldx, \boldy)\}_2 + \{u(\boldzero, \boldc, \boldb+2\boldx, \boldy)\}_2 + \{u(\boldzero, \boldb, \bolda+2\boldx, \boldy)\}_2
\end{split}
\end{equation}
for the triangulation of Theorem \ref{triangulation manin3}. The second line of \eqref{eq Tab} is invariant under cyclic permutation of $(\bolda, \boldb, \boldc)$. For the first line, we have
\begin{align*}
\sum_{\boldx \in (\Z/N\Z)^2} \{u(\boldzero, \boldx, \bolda-\boldx, \boldb+\boldx)\}_2 & = \sum_{\boldx \in (\Z/N\Z)^2} \{u(\boldzero, -\boldx, -\bolda+\boldx, \boldb+\boldx)\}_2 \\
& \stackrel{\boldx=\boldy-\boldb}{=} \sum_{\boldy \in (\Z/N\Z)^2} \{u(\boldzero, \boldb-\boldy, \boldc+\boldy, \boldy)\}_2 \\
& = \sum_{\boldy \in G} \{u(\boldzero, \boldy, \boldb-\boldy, \boldc+\boldy)\}_2
\end{align*}
since $\{1-1/f\}_2 = \{f\}_2$ by \cite[VI, Lemma 5.4]{Wei13}. Thus $T(\bolda, \boldb) = T(\boldb, \boldc) = T(\boldc, \bolda)$ and
\begin{equation*}
\tilde{\xi}(\bolda, \boldb) + \tilde{\xi}(\boldb, \boldc) + \tilde{\xi}(\boldc, \bolda) = T(\bolda, \boldb) \otimes \Bigl(\frac{g_\boldb}{g_\bolda} \frac{g_\boldc}{g_\boldb} \frac{g_\bolda}{g_\boldc} \Bigr) = 0.
\end{equation*}
The proof of $\tilde{\xi}(\boldb, \bolda) = \tilde{\xi}(\bolda, \boldb)$ for $N$ odd is similar.
\end{proof}

Numerical experiments suggest that the elements $\xi'(\bolda, \boldb)$ also satisfy the 2-term relations $\xi'(\bolda, \boldb) + \xi'(\boldb, -\bolda) = 0$, analogous to those for modular symbols, as well as $\xi'(-\bolda, \boldb) = \xi'(\bolda, -\boldb) = -\xi'(\bolda, \boldb)$, but we were not able to prove them---maybe another triangulation is needed. We also do not know whether these relations hold for the elements $\xi(\bolda, \boldb)$.

Proposition \ref{pro xiG 3-term} gives some hope to find an inductive procedure to construct elements in $H^2_\M(Y(N),\Q(n)) \cong K_{2n-2}^{(n)}(Y(N))$ for $n \geqslant 4$.

\section{Numerical computation of the regulator} \label{sec: num}

The aim of this section is to explain how to compute numerically certain regulator integrals associated to the elements $\xi(\bolda, \boldb)$ in $K_4^{(3)}(Y(N))$ and $\xi_1(a,b)$ in $K_4^{(3)}(Y_1(N))$. As an application, one can prove that these elements can be non-trivial (see Sections~\ref{sec: num beilinson} and~\ref{sec: comparison}). For this computation, the explicit lift of the $3$-term relation given in Theorem \ref{triangulation manin3} is required. We implemented this computation in PARI/GP \cite{PARI} and the program \texttt{K4-reg-num.gp} is available in the GitHub repository \cite{Bru22}.

We begin by recalling in Section \ref{subsec: reg} Goncharov's and De Jeu's theory of regulator maps \cite{Gon02, Jeu00} in our case of interest, and the link with Beilinson's regulator map. In Section~\ref{subsec: convergence}, we discuss in detail the convergence of Goncharov's integrals. Finally, we explain in Sections~\ref{subsec: mellin} and \ref{subsec modular reg} how to compute these integrals in the case of modular curves, using generalised Mellin transforms.

\subsection{Regulator on the polylogarithmic complex} \label{subsec: reg}

Goncharov has defined completely explicit regulator maps for complex algebraic varieties at the level of his polylogarithmic complexes \cite{Gon02}. We will use these regulator formulas in the case of curves.

We keep the same setting as in Section \ref{subsec: de jeu curves}: $X$ is a smooth projective geometrically connected curve defined over a number field $k$, and $F=k(X)$ is its function field. Let $Y = X \setminus S$, where $S$ is a finite set of closed points of $X$. We will write $Y(\C) = Y \times_{\Q} \C$ for the complex points of $Y$ (note that the Riemann surface $Y(\C)$ is not necessarily connected). It follows from \cite[Theorem 2.2]{Gon02} that we have a regulator map
\begin{equation*}
r_3(2) \colon H^2(\Gamma(Y,3)) \to H^1(Y(\C),\R(2))^+,
\end{equation*}
where $(\cdot)^+$ denotes the invariants with respect to complex conjugation acting on the second factor of $Y \times_{\Q} \C$. The map $r_3(2)$ is defined by Goncharov at the level of cochains by means of explicit differential forms, see \eqref{eq r32} for the precise formula.

It is expected that the regulator maps defined by Beilinson and Goncharov are compatible by means of De Jeu's map \eqref{eq De Jeu map K4}. More precisely, there should be a diagram, commuting up to sign,
\begin{equation} \label{eq reg diagram}
\begin{tikzcd}
H^2(\Gamma(Y,3)) \arrow{d} \arrow["r_3(2)", dr] & \\
K_4^{(3)}(Y) \arrow{r}{\frac12 r_B} & H^1(Y(\C),\R(2))^+,
\end{tikzcd}
\end{equation}
where $r_B$ is Beilinson's regulator map. The factor $\frac12$ comes from the comparison of \cite[Theorem 3.3]{Gon96}, where $c_3 = \frac43$, and \cite[(3.1)]{Jeu00}, where the rational factor is $\pm \frac83$.

De Jeu has proved a slightly weaker version of \eqref{eq reg diagram}, where the group $K_4^{(3)}(Y)$ is replaced by $K_4^{(3)}(Y)+(K_3^{(2)}(k) \cup F^\times)$, viewed inside $K_4^{(3)}(F)$, and $H^1(Y(\C),\R(2))^+$ is replaced by $\Hom_{\C}(\Omega^1(X(\C)), \C)$, see Theorem 2, Theorem 3.5 and Remark 3.7 in \cite{Jeu00}. This proves the existence and commutativity up to sign of \eqref{eq reg diagram} when $Y = X$ and $k$ is totally real \cite[Theorem 5.4]{Jeu00}.

Since our aim is to provide numerical evidence for Beilinson's conjecture and for Conjecture \ref{main conj}, we will not attempt to establish \eqref{eq reg diagram} rigorously in general.

\begin{remark}
The regulator map $r_3(2)$ can be used with $X=X(N)$ and $k=\Q(\zeta_N)$, or with $X=X_1(N)$ and $k=\Q$, the set $S$ being the set of cusps. However, the map $r_3(2)$ is defined for any Riemann surface $Y = X \setminus S$, where $X$ is compact connected and $S$ is a finite set of points. In the following Sections \ref{subsec: convergence}--\ref{subsec modular reg}, we will use this setting with the modular curve $Y = \Gamma \backslash \h$, where $\Gamma$ is a congruence subgroup of $\SL_2(\Z)$. Note that the map $\nu$ from Section~\ref{sec: uabcd} identifies $\Gamma(N) \backslash \h$ with a connected component of $Y(N)(\C)$, and the two regulator maps $r_3(2)$ are compatible via $\nu$. Therefore, we will be able to apply the results from these sections to the elements $\xi(\bolda, \boldb)$ and $\xi_1(a,b)$.
\end{remark}

\subsection{Convergence of Goncharov's integrals} \label{subsec: convergence}

We show in this section that Goncharov's regulator integrals are absolutely convergent in the case of Riemann surfaces (Proposition~\ref{conv int rn2}), and we apply this to modular curves (Corollary \ref{conv int rn2 modular}).

Let $X$ be a compact connected Riemann surface, and let $F$ be the field of meromorphic functions on $X$. Let $n \geqslant 3$. According to Goncharov's Conjecture \ref{Goncharov polylog conj}, the motivic cohomology group $H^2_\M(F,\Q(n)) \cong K_{2n-2}^{(n)}(F)$ is isomorphic to a certain subquotient of $B_{n-1}(F) \otimes F^\times_\Q$. Goncharov constructs in \cite[Theorem 2.2]{Gon02} a regulator map
\begin{equation*}
r_n(2) \colon B_{n-1}(F) \otimes F^\times_\Q \to \mathcal{A}^1(\eta_X)(n-1)
\end{equation*}
where $\mathcal{A}^1(\eta_X)(n-1)$ is the space of $(2\pi i)^{n-1} \R$-valued differential $1$-forms on $X$ which are regular outside a finite subset of $X$. Concretely, for $f \in F \backslash \{0,1\}$ and $g \in F^\times$, we have
\begin{equation} \label{def rn2}
\begin{split}
r_n(2)(\{f\}_{n-1} \otimes g) & = i \LL_{n-1}(f) \darg g - \frac{2^{n-1} \textbf{B}_{n-1}}{(n-1)!} \alpha(1-f,f) \cdot \log^{n-3} |f| \log |g| \\
& \qquad - \sum_{k=2}^{n-2} \frac{2^{k} \textbf{B}_{k}}{k!} \LL_{n-k}(f) \log^{k-2} |f| \dlog |f| \cdot \log |g|,
\end{split}
\end{equation}
where $\LL_m \colon \PP^1(\C) \to (2\pi i)^{m-1} \R$ is the single-valued polylogarithm defined in \cite[Section~2.1]{Gon02}\footnote{There is a misprint in \cite[Section 2.1]{Gon02}: $\log^{n-k} |z|$ should be replaced by $\log^k |z|$ in the definition of $\widehat{\mathcal{L}}_n(z)$.}, $\textbf{B}_k$ is the $k$th Bernoulli number, and
\begin{equation*}
\alpha(f,g) = - \log |f| \dlog |g| + \log |g| \dlog |f|.
\end{equation*}
In particular $\LL_2 = i D$, where $D \colon \PP^1(\C) \to \R$ is the Bloch-Wigner dilogarithm \cite[Section~2]{Zag91}. For $m \geqslant 2$, the function $\LL_m$ is real-analytic outside $\{0,1,\infty\}$ and is continuous on $\PP^1(\C)$ with $\LL_m(0)=\LL_m(\infty)=0$. It follows that the 1-form $r_n(2)(\{f\}_{n-1} \otimes g)$ is defined and real-analytic outside the set of zeros and poles of $f$, $1-f$ and $g$ in $X$.

For $n=3$, we have explicitly
\begin{equation} \label{eq r32}
r_3(2)(\{f\}_2 \otimes g) = -D(f) \darg g - \frac{1}{3} \alpha(1-f,f) \cdot \log |g|.
\end{equation}
This defines the map $r_3(2)$ in the diagram \eqref{eq reg diagram}.

We will use the big $O$ notation $f(x)=O(g(x))$ as $x \to x_0$ to indicate that there exists a constant $C>0$ such that $|f(x)| \leqslant C |g(x)|$ for all $x$ sufficiently close to $x_0$.

\begin{lemma} \label{estimate alpha}
Let $f, g$ be non-zero meromorphic functions on $X$. Let $z=re^{i\theta}$ be a local coordinate at a point $p \in X$. As $z \to 0$, we have
\begin{equation*}
\alpha(f,g) = \Bigl( -\frac{\log |\partial_p(f,g)|}{r} + O(\log r) \Bigr) dr + O(r \log r) d\theta
\end{equation*}
where $\partial_p(f,g) = (-1)^{\ord_p(f)\ord_p(g)} \bigl(f^{\ord_p(g)}/g^{\ord_p(f)}\bigr)(p)$ is the tame symbol of $f,g$ at $p$.
\end{lemma}

\begin{proof}
Write $f(z) = a z^m + O(z^{m+1})$ and $g(z) = b z^n + O(z^{n+1})$ with $a,b \in \C^\times$ and $m,n \in \Z$. A direct computation gives
\begin{equation*}
\dlog f = (m+O(z)) \frac{dz}{z} = (m + O(z)) \frac{dr}{r} + (im + O(z)) d\theta.
\end{equation*}
Taking the real and imaginary parts, we get
\begin{align}
\nonumber \dlog |f| & = \Bigl(\frac{m}{r} + O(1)\Bigr) dr +O(r) d\theta, \\
\label{estimate darg f} \darg f & = O(1) dr + (m+O(r)) d\theta.
\end{align}
On the other hand $\log |f| = \log |a| + m\log r +O(r)$. Putting things together, we arrive at
\begin{equation*}
\alpha(f,g) = \Bigl( -\frac{1}{r} \log \Bigl | \frac{a^n}{b^m} \Bigr | + O(\log r) \Bigl) dr + O(r\log r) d\theta. \qedhere
\end{equation*}
\end{proof}

\begin{proposition} \label{conv int rn2}
Let $f \in F \backslash \{0,1\}$ and $g \in F^\times$. Let $S$ be the set of zeros and poles of the functions $f$, $1-f$ and $g$ in $X$. Let $\gamma \colon [0,1] \to X$ be a $C^\infty$ path such that
\begin{enumerate}
\item The path $\gamma$ avoids $S$ except possibly at the endpoints;
\item If $p \in S$ is an endpoint of $\gamma$, then the argument of $\gamma(t)$ with respect to a local coordinate at $p$ is of bounded variation as $\gamma(t)$ approaches $p$.
\end{enumerate}
Then for every $n \geqslant 3$, the integral $\int_\gamma r_n(2)(\{f\}_{n-1} \otimes g)$ converges absolutely.
\end{proposition}

\begin{proof}
As noted above, the integrand is $C^\infty$ outside $S$. We are going to show the convergence of the integral at the endpoint $t=0$ (the case $t=1$ is identical). Let $z=re^{i\theta}$ be a local coordinate at $p=\gamma(0)$. Assumption (2) means that the form $d\theta$ is (absolutely) integrable along $\gamma$ near $t=0$. Moreover $dr = e^{-i\theta} dz  - i rd\theta$, so that $dr$ is also integrable along $\gamma$. Using~\eqref{estimate darg f}, we deduce that $\darg g$ and the first term of \eqref{def rn2} are integrable. Regarding the second term, Lemma \ref{estimate alpha} and the fact that $\partial_p(1-f,f)=1$ give
\begin{equation*}
\alpha(1-f,f) = O(\log r) dr + O(r\log r) d\theta.
\end{equation*}
It follows that the second term in \eqref{def rn2} has at worst logarithmic singularities, hence is integrable. Finally, the integrability of the third term in \eqref{def rn2} can be proved similarly, noting that $\LL_m(z) = O(|z| \log^{m-1} |z|)$ as $z \to 0$ for any $m \geqslant 2$, and using the functional equation $\LL_m(1/z) = (-1)^{m-1} \LL_m(z)$ to get the asymptotic behaviour as $z \to \infty$.
\end{proof}

In the case of modular curves, this has the following consequence. For any two cusps $\alpha_1 \neq \alpha_2$ in $\PP^1(\Q)$, we denote by $\{\alpha_1 \to \alpha_2\}$ the hyperbolic geodesic from $\alpha_1$ to $\alpha_2$ in $\h$.

\begin{corollary} \label{conv int rn2 modular}
Let $\Gamma$ be a congruence subgroup of $\SL_2(\Z)$. Let $X = \Gamma \backslash (\h \cup \PP^1(\Q))$ be the compactification of $\Gamma \backslash \h$. Let $u,v$ be modular units for $\Gamma$ such that $1-u$ is also a modular unit. Then for any $n \geqslant 3$ and any two cusps $\alpha_1 \neq \alpha_2$ in $\PP^1(\Q)$, the integral of $r_n(2)(\{u\}_{n-1} \otimes v)$ along $\{\alpha_1 \to \alpha_2\}$ converges absolutely.
\end{corollary}

Proposition \ref{conv int rn2} also holds for $n=2$; in fact in this case we don't need to include the function $1-f$ in the definition of $S$. This follows from a similar computation, the regulator being defined by $r_2(2) (f \wedge g) = -i (\log |f| \darg g - \log |g| \darg f)$. As a consequence, Corollary~\ref{conv int rn2 modular} also holds in the case $n=2$, without assuming that $1-u$ is a modular unit.

\subsection{Generalised Mellin transforms} \label{subsec: mellin}

We wish to compute numerically the integral
\begin{equation} \label{eq int rn2uv}
\int_{\alpha_1}^{\alpha_2} r_n(2)(\{u\}_{n-1} \otimes v)
\end{equation}
from Corollary \ref{conv int rn2 modular}. We work in Sections \ref{subsec: mellin}--\ref{subsec modular reg} with the group $\Gamma = \Gamma(N)$. Let $X = \Gamma(N) \backslash (\h \cup \PP^1(\Q))$ be the compactification of $\Gamma(N) \backslash \h$. By Manin's theorem \cite{Man72}, the first homology group of $X$ relative to the cusps is generated by the modular symbols $\{h0 \to h\infty\}$ with $h \in \SL_2(\Z)$. In this section, we will only compute \eqref{eq int rn2uv} for such paths. The general case requires some more care, due to the fact that the integrand may have non-trivial residues at the cusps, so that the formula $\int_{\alpha_1}^{\alpha_2} + \int_{\alpha_2}^{\alpha_3} = \int_{\alpha_1}^{\alpha_3}$ does not always hold.

Let $u,v$ be modular units for $\Gamma(N)$ such that $1-u$ is also a modular unit, and let $n \geqslant 3$. Note that
\begin{equation} \label{eq int rn2 0 infty}
\int_{h 0}^{h \infty} r_n(2)(\{u\}_{n-1} \otimes v) = \int_0^\infty r_n(2)(\{u \circ h\}_{n-1} \otimes (v \circ h)),
\end{equation}
and the functions $u \circ h$ and $v \circ h$ are also modular units for $\Gamma(N)$. We are thus reduced to the case $h$ is the identity matrix $I_2$, so that we integrate from $0$ to $\infty$. In this case, we restrict the $1$-form $r_n(2)(\{u\}_{n-1} \otimes v)$ to the imaginary axis in $\h$, and write
\begin{equation} \label{eq int phi}
\int_0^\infty r_n(2)(\{u\}_{n-1} \otimes v) = \int_0^\infty \phi(y) dy,
\end{equation}
where $\phi \colon ]0,+\infty[ \to \C$ is a $C^\infty$ function. We have seen in Corollary \ref{conv int rn2 modular} that $\phi$ is absolutely integrable. We are going to demonstrate that such an integral can be computed rapidly for a specific class of functions. We will show later, in Section \ref{subsec modular reg}, that the function $\phi$ in \eqref{eq int phi} belongs to this class.

\begin{definition}
Let $\mathcal{P}$ be the class of $C^\infty$ functions $\phi \colon ]0,+\infty[ \to \C$ such that
\begin{equation} \label{eq class P}
\phi(y) = \sum_{j=0}^{j_\infty} y^j \sum_{n = 0}^{\infty} a_n^{(j)} e^{-2\pi ny/N}, \quad \textrm{and} \quad \phi\Big(\frac{1}{y}\Bigr) = \sum_{j=0}^{j_0} y^j \sum_{n = 0}^{\infty} b_n^{(j)} e^{-2\pi ny/N},
\end{equation}
for some integers $j_\infty, j_0 \geqslant 0$, and where the sequences $(a_n^{(j)})_{n \geqslant 0}$ and $(b_n^{(j)})_{n \geqslant 0}$ have polynomial growth as $n \to \infty$.
\end{definition}

By considering the asymptotic expansion, it is easy to see that the coefficients $a_n^{(j)}$ and $b_n^{(j)}$ are uniquely determined by $\phi$. Moreover, the function $\phi$ is absolutely integrable on $]0,+\infty[$ if and only if $a_0^{(j)}=0$ for $j \geqslant 0$ and $b_0^{(j)}=0$ for $j \geqslant 1$.

We recall briefly the definition of Mellin transform (see \cite[Section 3.4]{Car92} for details). The (generalised) Mellin transform $\mathcal{M}(\phi,s)$ of a function $\phi \in \mathcal{P}$ is defined by
\begin{align*}
\mathcal{M}(\phi,s) = \int_0^\infty \phi(y) y^s \frac{dy}{y} & := \textrm{a.c.} \Bigl(\int_1^\infty \phi(y) y^s \frac{dy}{y}\Bigr) + \textrm{a.c.} \Bigl(\int_0^1 \phi(y) y^s \frac{dy}{y}\Bigr) \\
& = \textrm{a.c.} \Bigl(\int_1^\infty \phi(y) y^s \frac{dy}{y}\Bigr) + \textrm{a.c.} \Bigl(\int_1^\infty \phi\Bigl(\frac{1}{y}\Bigr) y^{-s} \frac{dy}{y}\Bigr),
\end{align*}
where $s \in \C$ and ``a.c.'' means analytic continuation with respect to $s$. Note that the first integral converges for $\Re(s) \ll 0$ while the second integral converges for $\Re(s) \gg 0$; both have a meromorphic continuation to $s \in \C$ thanks to \eqref{eq class P}.

In the case $\phi$ is absolutely integrable, the integral $\int_0^\infty \phi = \mathcal{M}(\phi,1)$ can be computed by integrating term by term the expansions \eqref{eq class P}. This yields the following series with exponential decay:
\begin{equation} \label{eq mellin phi}
\int_0^\infty \phi = \sum_{n = 1}^\infty \sum_{j=0}^{j_\infty} a_n^{(j)} \Bigl(\frac{N}{2\pi n}\Bigr)^{j+1} \Gamma\Bigl(j+1, \frac{2\pi n}{N}\Bigr) + \sum_{n = 1}^\infty \sum_{j=0}^{j_0} b_n^{(j)} \Bigl(\frac{N}{2\pi n}\Bigr)^{j-1} \Gamma\Bigl(j-1, \frac{2\pi n}{N}\Bigr) + b_0^{(0)},
\end{equation}
where $\Gamma(s,x) = \int_x^\infty e^{-y} y^{s-1} dy$ is the incomplete gamma function, which can be computed efficiently in PARI/GP. So the integral $\int_0^\infty \phi$ can be computed accurately, provided sufficiently many coefficients $a_n^{(j)}$ and $b_n^{(j)}$ are known.

The class $\mathcal{P}$ is a $\C$-algebra stable under the differentiation $\frac{d}{dy}$. However it is not stable under taking primitive: for example, the function $\phi = 1$ does not have a primitive in $\mathcal{P}$. In fact, a function $\phi \in \mathcal{P}$ has a primitive in $\mathcal{P}$ if and only if $b_n^{(0)} = b_n^{(1)} = 0$ for all $n \geqslant 0$. This criterion shows that the image $\mathcal{P}'$ of the operator $\frac{d}{dy} \colon \mathcal{P} \to \mathcal{P}$ is an ideal of $\mathcal{P}$.

\subsection{The modular case} \label{subsec modular reg}

We show in this section that in the case of modular curves, the regulator integrals defined by Goncharov belong to $\mathcal{P}$, and we explain how to compute them. We say that a function $\Phi \colon \h \to \C$ is in $\mathcal{P}$ if the function $y \mapsto \Phi(iy)$ is in $\mathcal{P}$.

\begin{lemma} \label{log u in P}
For any modular unit $u$ for $\Gamma(N)$, we have $\log u \in \mathcal{P}$, were $\log u$ is any determination of the logarithm of $u$ on $\h$. In particular $\log |u| \in \mathcal{P}$, and the forms $\dlog |u|$ and $\darg u$ belong to $d\mathcal{P}$.
\end{lemma}

\begin{proof}
Since the group $\mathcal{O}(Y(N))^\times_\Q/\Q(\zeta_N)^\times_\Q$ is generated by the Siegel units, it suffices to prove the result for them. For the asymptotic expansion of $\log g_\boldx(iy)$ as $y \to +\infty$, this follows from taking the logarithm of \eqref{eq def gx} and expanding as a power series in $e^{-2\pi y/N}$. The expansion as $y \to 0$ also has the correct shape because of the relation $g_\boldx \circ \sigma = g_{\boldx \sigma}$ in $\mathcal{O}(Y(N))^\times_\Q$ for the matrix $\sigma = \sabcd{0}{-1}{1}{0}$ (see the formula before Proposition \ref{pro uabcd transform}).
\end{proof}

\begin{proposition} \label{Ln u in P}
Let $u$ be a modular unit for $\Gamma(N)$ such that $1-u$ is also a modular unit. For every $n \geqslant 2$, we have $\LL_n(u) \in \mathcal{P}$.
\end{proposition}

\begin{proof}
We prove this by complete induction on $n$. For $n=2$, we have $\LL_2 = i D$, where $D$ is the Bloch-Wigner dilogarithm. By definition,
\begin{equation*}
D(z) = \Im \Bigl( \sum_{n = 1}^\infty \frac{z^n}{n^2} \Bigr) + \log |z| \arg(1-z) \qquad (|z| < 1)
\end{equation*}
and a computation gives
\begin{equation*}
d D(u) = \log |u| \darg (1-u) - \log |1-u| \darg (u).
\end{equation*}
From Lemma \ref{log u in P}, it follows that $dD(u) \in d\mathcal{P}$, hence $D(u) \in \mathcal{P}$.

Now let $n \geqslant 3$. By the commutative diagram in \cite[Theorem 2.2]{Gon02}, we have
\begin{align*}
d\LL_n(u) & = r_n(2)(\{u\}_{n-1} \otimes u) \\
& = i \LL_{n-1}(u) \darg u - \frac{2^{n-1} \textbf{B}_{n-1}}{(n-1)!} \alpha(1-u,u) \cdot \log^{n-2} |u|\\
& \qquad - \sum_{k=2}^{n-2} \frac{2^{k} \textbf{B}_{k}}{k!} \LL_{n-k}(u) \log^{k-1} |u| \dlog |u|
\end{align*}
By the induction hypothesis $\LL_m(u)$ belongs to $\mathcal{P}$ for $m<n$. The result now follows from Lemma \ref{log u in P} and the fact that $\mathcal{P}'$ is an ideal of $\mathcal{P}$.
\end{proof}

The proof of Proposition \ref{Ln u in P} provides a way to compute the expansions \eqref{eq class P} of $\LL_n(u)$ inductively on $n$: we first compute the expansions of $d\LL_n(u)$ at $0$ and $\infty$, and then integrate the results. The constant of integration (in other words, the coefficient $a_0^{(0)}$ in the expansion) is determined by computing $\LL_n(u(\infty))$.

\begin{theorem}
Let $n \geqslant 3$. Let $u,v$ be two modular units for $\Gamma(N)$ such that $1-u$ is also a modular unit. Write
\begin{equation*}
r_n(2)(\{u\}_{n-1} \otimes v) |_{\{0,\infty\}} = \phi(y) dy.
\end{equation*}
Then $\phi$ belongs to $\mathcal{P}$.
\end{theorem}

\begin{proof}
This follows from \eqref{def rn2}, Lemma \ref{log u in P} and Proposition \ref{Ln u in P}.
\end{proof}

The integral \eqref{eq int rn2 0 infty} can then be computed numerically in the case $u=u(\bolda, \boldb, \boldc, \boldd)$ and $v=g_{\boldsymbol e}$ as follows (note that $v$ is only a root of a modular unit, but the regulator integral still makes sense). With Proposition \ref{pro uabcd transform}, we reduce to the case $h=I_2$. Using Definition~\ref{def gx}, Proposition \ref{pro uabcd gx} and Proposition \ref{pro uabcd transform}, we compute the asymptotic expansions of $\log(u)$, $\log(1-u)$ and $\log(v)$ at $0$ and $\infty$, for some determination of these functions in $\h$. The formula \eqref{def rn2} and the proof of Proposition \ref{Ln u in P} provide, as explained above, a way to compute the asymptotic expansions of $r_n(2)(\{u\}_{n-1} \otimes v)$ at $0$ and $\infty$. We then obtain the associated regulator integral by \eqref{eq mellin phi}.

Finally, regarding the regulator of the elements $\xi(\bolda, \boldb)$, we fix the representative $\tilde{\xi}(\bolda, \boldb)$ as in Definition \ref{def tilde xi ab}, and consider the $1$-form $r_3(2)(\tilde{\xi}(\bolda, \boldb))$ on $\h$. Using the above method, we can compute the (absolutely convergent) integral of this $1$-form along the imaginary axis $\{0 \to \infty\}$, and more generally along any modular symbol $\{h0 \to h\infty\}$ with $h \in \SL_2(\Z)$. The same method works for the elements $\xi_1(a,b)$, using the representative \eqref{eq xi1tilde}.

This algorithm works for any weight $n \geqslant 2$, but we implemented it in PARI/GP only in the case $n=3$ and for the elements $\xi_1(a,b)$, see the script $\texttt{K4-reg-num.gp}$ in the GitHub repository \cite{Bru22}.

\section{Numerical verification of the Beilinson conjecture} \label{sec: num beilinson}

In this section, we check numerically Beilinson's conjecture on $L(E,3)$ for elliptic curves $E$ over $\Q$, where $L(E,s)$ is the $L$-function of $E$. 
This conjecture relates $L(E,3)$ to the Beilinson regulator map on the group $K_4^{(3)}(E)$ \cite[Section 4]{Jeu96}. Beilinson \cite{Bei86} has proved a weak form of his conjecture for modular elliptic curves by constructing elements in $K_4^{(3)}(E)$ using his theory of the Eisenstein symbol. Here, we do only a numerical comparison, but we use the elements $\xi_1(a,b)$ constructed in Section \ref{def xiab xi1ab}, which are of different nature than the Beilinson elements. Our computation implies that the elements $\xi_1(a,b)$ are generally non-zero. Let us mention that De Jeu \cite[Section~6]{Jeu96} also constructed explicit elements in $K_4^{(3)}(E)$ for a specific elliptic curve $E$ of conductor $20$, and checked numerically Beilinson's conjecture on $L(E,3)$ for this curve.

Our approach differs from \cite{Jeu96, Jeu00, DJZ06} in the following way. In  \cite{Jeu96, Jeu00}, the regulator $1$-form is integrated against a holomorphic $1$-form, which requires to compute a two-dimensional integral on $E(\C)$. In \cite{DJZ06}, the regulator $1$-form is integrated along $1$-cycles, which is computationally easier. In our method, we also compute a one-dimensional integral, but we take advantage of the fact that we have Fourier expansions at our disposal.

Let us recall the weak form of Beilinson's conjecture on $L(E,3)$. Let $E$ be an elliptic curve of conductor $N$. Let $r_B : K_4^{(3)}(E) \to H^1(E(\C), \R(2))^+$ denote the Beilinson regulator map, where $(\cdot)^+$ is the subspace of invariants under complex conjugation. Let $\gamma_E$ be a generator of $H_1(E(\C), \Q)^+$. Then Beilinson's conjecture asserts that there exists $\xi_E \in K_4^{(3)}(E)$ such that
\begin{equation} \label{conj Beilinson E}
\int_{\gamma_E} r_B(\xi_E) = r \pi^2 L'(E,-1) \qquad (r \in \Q^\times).
\end{equation}
Note that by the functional equation of $L(E,s)$ \cite[Appendix C, Theorem 16.3]{Sil09}, we have $L'(E,-1) = \pm \frac{N^2}{8 \pi^4} L(E,3)$.

Let us outline our strategy. Let $\varphi \colon X_1(N) \to E$ be a modular parametrisation. We consider the elements $\xi'_1(a,b)$ of $K_4^{(3)}(X_1(N))$ and, as in Beilinson's approach, we apply to them the trace map $\varphi_* \colon K_4^{(3)}(X_1(N)) \to K_4^{(3)}(E)$. Let $\gamma_{N,E}$ be a generator of the Hecke eigenspace of $H_1(X_1(N)(\C), \Q)^+$ associated to $E$. Then, to show \eqref{conj Beilinson E}, it is enough to find $a,b \in \Z/N\Z$ such that
\begin{equation} \label{conj Beilinson X1N}
\int_{\gamma_{N,E}} r_B(\xi'_1(a,b)) = r \pi^2 L'(E,-1) \qquad (r \in \Q^\times).
\end{equation}
It is this task that we perform numerically.

The element $\gamma_{N,E}$ can be computed using the implementation of modular symbols in Magma. More precisely, we obtain an explicit 1-cycle $\tilde{\gamma}_{N,E}$ on $X_1(N)(\C)$, such that the homology class of $\tilde{\gamma}_{N,E}$ belongs to the Hecke eigenspace associated to $E$, and whose image $\gamma_{N,E}$ under the projection map $H_1(X_1(N)(\C), \Q) \to H_1(X_1(N)(\C), \Q)^+$ is non-zero. The cycle $\tilde{\gamma}_{N,E}$ is given as a $\Z$-linear combination of modular symbols $\{h0 \to h\infty\}$ with $h \in \SL_2(\Z)$. The Magma code and the resulting cycles $\tilde{\gamma}_{N,E}$ for all isogeny classes of elliptic curves $E$ of conductor up to 100 are available in the GitHub repository \cite{Bru22}, see the files $\texttt{HomologyBasis.m}$ and $\texttt{DataH1ell}$, respectively.

For the computation of the integral \eqref{conj Beilinson X1N}, we restrict to the simpler situation where $\xi_1(a,b)$ has trivial residues at the cusps, so that $\xi'_1(a,b) = \xi_1(a,b)$ (note that this is automatically true if $N$ is prime or twice a prime by Proposition \ref{pro N prime}). We have the following lemma.

\begin{lemma} \label{lem rB r32}
Assume that $\xi_1(a,b)$ has trivial residues at the cusps, and write $\tilde{\gamma}_{N,E} = \sum n_i \{h_i 0 \to h_i \infty\}$. Then
\begin{equation} \label{eq rB r32}
\int_{\gamma_{N,E}} r_B(\xi'_1(a,b)) = \pm 2 \sum n_i \int_{h_i 0}^{h_i \infty} r_3(2)(\tilde{\xi}_1(a,b)).
\end{equation}
\end{lemma}

\begin{proof}
Since $r_B(\xi'_1(a,b))$ is invariant under the action of complex conjugation, we have
\begin{equation*}
\int_{\gamma_{N,E}} r_B(\xi'_1(a,b)) = \int_{[\tilde{\gamma}_{N,E}]} r_B(\xi'_1(a,b)).
\end{equation*}
By the diagram \eqref{eq reg diagram}, the restriction of $r_B(\xi'_1(a,b))$ to $Y_1(N)(\C)$ is represented by the differential form $\eta = \pm 2 r_3(2)(\tilde{\xi}_1(a,b))$, and $\eta$ has trivial residues at the cusps. Note that $\eta$ is defined only on $Y_1(N)(\C)$, while the cycle $\tilde{\gamma}_{N,E}$ passes through cusps. We modify this cycle slightly as shown in Figure \ref{modify gamma}.

\begin{figure}[h]
    \centering
\begin{tikzpicture}[scale=1.25]
\draw [-{Stealth[length=2mm]}] (0,0)--(1,0.7);
\draw (1,0.7) node[anchor=north west] {$\tilde{\gamma}_{N,E}$};
\draw (1,0.7)--(2,1.4);
\draw [-{Stealth[length=2mm]}] (2,1.4)--(3,0.7);
\draw (3,0.7)--(4,0);
\filldraw (2,1.4) circle (2pt);
\draw (2,1.4) node[anchor=south] {cusp};
\path [name path=gamma1] (6,0)--(8,1.4);
\path [name path=loop] (8,1.4) circle (.7cm);
\path [name path=gamma2] (8,1.4)--(10,0);
\draw [-{Stealth[length=2mm]}] (6,0)--(7,0.7);
\draw [name intersections={of=gamma1 and loop, by={A1}}] (7,0.7)--(A1);
\draw [-{Stealth[length=2mm]}] [name intersections={of=loop and gamma2, by={A2}}] (A2)--(9,0.7);
\draw (9,0.7)--(10,0);
\draw [dotted] (8,1.4) circle (.7cm);
\draw [<->] (8,1.4)--(8,2.1);
\draw (8,1.75) node[anchor=west] {$\varepsilon$};
\draw [-{Stealth[length=2mm]}] (A1) arc [start angle=-145, end angle=-90, radius=.7cm];
\draw (8,0.7) node[anchor=north] {$\gamma_\varepsilon$};
\draw (8,0.7) arc [start angle=-90, end angle=-35, radius=.7cm];
\end{tikzpicture}
    \caption{Modifying the cycle $\tilde{\gamma}_{N,E}$ at a cusp}
\label{modify gamma}
\end{figure}
The modified cycle $\gamma_{\varepsilon}$ (depending on a radius $\varepsilon$ for a fixed choice of local coordinate at each cusp) is homologous to $\tilde{\gamma}_{N,E}$ and is contained in $Y_1(N)(\C)$, so that
\begin{equation*}
\int_{[\tilde{\gamma}_{N,E}]} r_B(\xi'_1(a,b)) = \int_{\gamma_\varepsilon} \eta.
\end{equation*}
Since $\eta$ has trivial residues at the cusps, this integral converges to the right-hand side of \eqref{eq rB r32} as $\varepsilon \to 0$.
\end{proof}

Thanks to Lemma \ref{lem rB r32}, the integral \eqref{conj Beilinson X1N} can be computed using the method explained at the end of Section \ref{sec: num}.

It remains to find a way to determine whether $\xi_1(a,b)$ has trivial residues. By Theorem \ref{thm DJ residues}, this condition is equivalent to $\tilde{\xi}_1(a,b)$ having trivial residues at the cusps. The residue of the symbol $\tilde{\xi}_1(a,b)$ at a cusp $x \in X_1(N)$ is defined in Section \ref{subsec: de jeu curves}.
Let $\sigma_1, \overline{\sigma}_1, \ldots, \sigma_{r_2}, \overline{\sigma}_{r_2}$ be the non-real embeddings of $\Q(x)$. The Bloch-Wigner dilogarithm $D \colon \PP^1(\C) \to \R$ induces by linearity a map
\begin{equation*}
\tilde{D} \colon B_2(\Q(x)) \to \R^{r_2}, \quad \{a\}_2 \mapsto (D(\sigma_i(a)))_{1 \leqslant i \leqslant r_2}.
\end{equation*}
The restriction of $\tilde{D}$ to $H^1(\Gamma(\Q(x),2))$ is injective, because the composition of $\tilde{D}$ with Suslin's isomorphism $K_3^{(2)}(\Q(x)) \cong H^1(\Gamma(\Q(x),2))$ is a rational multiple of Borel's regulator map on $K_3^{(2)}(\Q(x))$ (see \cite[Section 2]{Gon95}). So $\xi_1(a,b)$ has trivial residue at $x$ if and only if $\tilde{D} \circ \Res_x (\tilde{\xi}_1(a,b)) = 0$. This condition can be checked numerically, since the cocycle $\tilde{\xi}_1(a,b)$ is explicit and the Bloch-Wigner dilogarithm can be computed in PARI/GP. We thus have a credible way to detect whether the residues of $\xi_1(a,b)$ at the cusps are trivial or not.

In practice, we did the following. For each elliptic curve $E$ of conductor $N \leqslant 50$, we searched for $a,b \in \Z/N\Z$ satisfying the following two conditions:
\begin{enumerate}
\item The residues of $\xi_1(a,b)$ at the cusps of $X_1(N)$ are (numerically) zero;
\item The integral \eqref{conj Beilinson X1N} is non-zero.
\end{enumerate}
For the first pair $(a,b)$ found, the integral \eqref{conj Beilinson X1N} always turned out to be numerically a rational multiple of $\pi^2 L'(E,-1)$, as predicted by Beilinson's conjecture.

\begin{theorem} \label{thm reg num}
For every elliptic curve $E$ of conductor $N \leqslant 50$, there exist $a,b \in \Z/N\Z$ such that the residues of $\xi_1(a,b)$ at the cusps of $X_1(N)$ are numerically zero, and
\begin{equation} \label{eq reg num}
\int_{\tilde{\gamma}_{N,E}} r_3(2) (\tilde{\xi}_1(a,b)) \stackrel{?}{=} \frac{r \pi^2}{N} L'(E,-1)
\end{equation}
where $r \in \Q^\times$ is a non-zero rational number of small height, and $\stackrel{?}{=}$ means equality to at least 40 decimal places.
\end{theorem}

Curiously, for our choice of $\tilde{\gamma}_{N,E}$ and $(a, b)$, the rational factor $r$ in \eqref{eq reg num} was always $\pm 3$, except for the curves $38a1$ ($r=9$), $42a1$ ($r=6$) and $43a1$ ($r=-6$). We don't have an explanation for this. Also, for the elliptic curve $E=\textrm{36a1}$ and the pair $(a,b)=(1,4)$, some residues are non-trivial, yet the integral is non-zero and proportional to $L'(E,-1)$. So having trivial residues at the cusps is not always necessary to get \eqref{eq reg num}.

Theorem \ref{thm reg num} was obtained using the function $\texttt{checkBeilinson}$ in the script \texttt{K4-reg-num.gp} in \cite{Bru22}. The results are stored in the file $\texttt{checkBeilinson50.txt}$. The computation took approximately 7 hours on a standard desktop computer.

\begin{example} \label{ex7}
For the elliptic curve $E \cong X_1(15)$ of Example \ref{ex1}, we find
\begin{equation*}
\int_{\tilde{\gamma}_{15,E}} r_3(2) (\tilde{\xi}_1(1,4)) \stackrel{?}{=} \frac{\pi^2}{5} L'(E,-1) = -\frac{45}{8\pi^2} L(E,3).
\end{equation*}
We also know from Example \ref{ex5} that $\xi_1(1,4)$ belongs to $K_4^{(3)}(X_1(15)) \cong K_4^{(3)}(E)$, so by Lemma \ref{lem rB r32}, we have
\begin{equation*}
\int_{\gamma_{15,E}} r_B (\xi_1(1,4)) \stackrel{?}{=} \pm \frac{2\pi^2}{5} L'(E,-1).
\end{equation*}
One can show in this case that $2 \gamma_{15,E}$ corresponds to a generator of $H_1(E(\C), \Z)^+$.
\end{example}

\section{Comparison with the Beilinson elements} \label{sec: comparison}

In this section, we aim to compare the elements $\xi_1(a,b)$ with the Beilinson elements in $K_4^{(3)}(Y_1(N))$. To this end, we compute numerically certain regulator integrals associated to both elements. These computations suggest that the elements $\xi_1(a,b)$ and the Beilinson elements are in fact proportional (Conjecture \ref{conj xi1 Eis001}).

One word of caution is in order here. The Beilinson regulator map defined on $K_4^{(3)}(Y_1(N))$ takes values in the de Rham cohomology group $H^1(Y_1(N)(\C), \R)$, while the modular symbols live in the relative homology group $H_1(X_1(N)(\C), S, \Z)$, where $S$ is the set of cusps. The integration of a form along a chain does not give a well-defined pairing between these two groups. Therefore, the integrals considered here are a priori not cohomological, and we have to fix representatives of the homology and cohomology classes. It would be interesting to find a conceptual framework to deal with this issue. The situation is a bit similar to that of periods of Eisenstein series \cite[Section 8]{PP13}, and Stevens's theory of extended modular symbols \cite{Ste89} may be helpful in this regard.

Let us recall the definition of the Beilinson elements (see \cite{Bei86} and \cite[Chapter 2]{Wan20}). For $N \geqslant 3$, let $p \colon E(N) \to Y(N)$ be the universal elliptic curve. In \cite[Section 3]{Bei86}, Beilinson constructs the Eisenstein symbol
\begin{equation*}
\Eis^1(\bolda) \in K_2^{(2)}(E(N)) \qquad (\bolda \in (\Z/N\Z)^2).
\end{equation*}
Following Beilinson, we define, for $\bolda, \boldb \in (\Z/N\Z)^2$,
\begin{equation} \label{def Eis001}
\Eis^{0,0,1}(\bolda, \boldb) = p_* (\Eis^1(\bolda) \cup \Eis^1(\boldb)) \in K_4^{(3)}(Y(N)).
\end{equation}
For $a,b \in \Z/N\Z$, we put
\begin{equation*}
\Eis_1^{0,0,1}(a,b) = \Eis^{0,0,1}((0,a), (0,b)).
\end{equation*}
Because of the transformation formula \cite[Lemma 2.3.5]{Wan20}
\begin{equation} \label{eq Eis001 GL2}
\Eis^{0,0,1}(\bolda, \boldb) | \gamma = \Eis^{0,0,1}(\bolda \gamma, \boldb \gamma) \qquad (\gamma \in \GL_2(\Z/N\Z)),
\end{equation}
the element $\Eis_1^{0,0,1}(a,b)$ defines an element of $K_4^{(3)}(Y_1(N))$.

The regulator computation goes as follows. We first use Magma to find a family $(\gamma_i)_i$ of $1$-cycles on $X_1(N)(\C)$ whose images under the projection map $H_1(X_1(N)(\C),\Q) \to H_1(X_1(N)(\C),\Q)^+$ form a basis of the latter vector space. Each $\gamma_i$ is given as a $\Z$-linear combination of modular symbols $\{h0 \to h\infty\}$ with $h \in \SL_2(\Z)$. The Magma code and the resulting data for $N \leqslant 50$ are available in the GitHub repository \cite{Bru22}, see the files $\texttt{HomologyBasis.m}$ and $\texttt{DataH1}$, respectively.
The integral of the form $r_3(2)(\tilde{\xi}_1(a,b))$ along $\gamma_i$ is then computed using the method explained at the end of Section \ref{sec: num}.

Regarding the Beilinson elements, we consider as in \cite[Proposition 2.4.2]{Wan20} a specific differential form $\Eis^{0,0,1}_{\mathcal{D}}(\bolda, \boldb)$ on $Y(N)(\C)$ representing $r_B(\Eis^{0,0,1}(\bolda, \boldb))$. Wang \cite[Theorem 0.1.3]{Wan20} computed theoretically the integral of $\nu^* \Eis^{0,0,1}_{\mathcal{D}}(\bolda, \boldb)$ along $\{0 \to \infty\}$. The formula involves the $L$-value of a product of two Eisenstein series of weight $1$. Thanks to \eqref{eq Eis001 GL2}, we can evaluate the integral of $\Eis_{1,\mathcal{D}}^{0,0,1}(a,b)$ along $\{h0 \to h\infty\}$ for any $h \in \SL_2(\Z)$. For simplicity, we state Wang's result only in the case $h=I_2$.

\begin{theorem}[Wang]
For any $a, b$ in $(\Z/N\Z) \backslash \{0\}$, we have
\begin{equation*}
\int_0^\infty \Eis^{0,0,1}_{1, \mathcal{D}}(a,b) = - \frac{36 \pi^2}{N^3} L'(\tilde{s}_a \tilde{s}_b, -1),
\end{equation*}
where the $\tilde{s}_x$ are Eisenstein series of weight 1 and level $\Gamma_1(N)$ defined by
\begin{equation*}
\tilde{s}_x(\tau) = \frac12 - \Bigl\{ \frac{x}{N} \Bigr\} + \sum_{\substack{m,n \geqslant 1 \\ n \equiv x \bmod{N}}} q^{mn} - \sum_{\substack{m,n \geqslant 1 \\ n \equiv -x \bmod{N}}} q^{mn} \qquad (q=e^{2\pi i\tau}),
\end{equation*}
and $\{\cdot\}$ denotes the fractional part.
\end{theorem}

In fact, Wang proves a much more general statement concerning the regulators of motivic classes $\Eis^{k_1,k_2,j}(\bolda, \boldb)$ on Kuga-Sato varieties \cite[Chapter 6]{Wan20}. These classes were defined by Beilinson, Deninger--Scholl \cite[5.7]{DS91} and Gealy \cite{Gea05}. The formula involves the completed $L$-function of a modular form of weight $k_1+k_2+2$ evaluated at $s=-j$.

The Eisenstein series $\tilde{s}_x$ appear in the work of Borisov and Gunnells \cite[3.18]{BG01}. Moreover, the transform of $\tilde{s}_x$ under the Atkin-Lehner involution $W_N \colon \tau \mapsto -1/N\tau$ is a multiple of the Eisenstein series denoted by $s_x$ in \cite[3.5]{BG01}. We can then use the standard recipe \cite[eq.~(4.3.13)]{Miy06} to compute numerically the $L$-value $L'(\tilde{s}_a \tilde{s}_b, -1)$, and thus the regulator integral. The computation for an arbitrary modular symbol $\{h0 \to h\infty\}$ proceeds similarly. We obtain the following theorem.

\begin{theorem} \label{thm comparison xi1 Eis001}
For every $N \leqslant 28$, every $a,b \in \Z/N\Z$ and every cycle $\gamma_i$, we have
\begin{equation*}
\int_{\gamma_i} r_3(2)(\tilde{\xi}_1(a,b)) \stackrel{?}{=} \frac{N^2}{6} \int_{\gamma_i} \Eis^{0,0,1}_{1,\mathcal{D}}(a,b),
\end{equation*}
where $\stackrel{?}{=}$ means equality to at least 40 decimal places.
\end{theorem}

Theorem \ref{thm comparison xi1 Eis001} was obtained using the function $\texttt{checkElements}$ in the script \texttt{K4-reg-num.gp} in \cite{Bru22}. The results are stored in the file $\texttt{checkElements28.txt}$. The computation took approximately 4 days on a standard desktop computer.

Based on Theorem \ref{thm comparison xi1 Eis001} and on the diagram \eqref{eq reg diagram}, we formulate the following conjecture.

\begin{conjecture} \label{conj xi1 Eis001}
For every $a, b \in \Z/N\Z$, we have in $K_4^{(3)}(Y_1(N))$:
\begin{equation*}
\xi_1(a, b) = \pm \frac{N^2}{3} \Eis_1^{0,0,1}(a, b).
\end{equation*}
\end{conjecture}

There is now theoretical evidence for this conjecture. Namely, for any $\bolda, \boldb \in (\Z/N\Z)^2$ such that the coordinates of $\bolda$, $\boldb$, $\bolda+\boldb$ are non-zero, the integral of $\nu^* r_3(2)(\tilde{\xi}(\bolda, \boldb))$ from $0$ to $\infty$ has been computed in terms of special values of $L$-functions of modular forms of weight~$2$ at $s=3$ \cite{BZ23}. The formula obtained is consistent with Wang's theorem, which leads us to expect that Conjecture \ref{conj xi1 Eis001} extends to the elements $\xi(\bolda, \boldb)$:

\begin{conjecture} \label{conj xi Eis001}
For every $\bolda, \boldb \in (\Z/N\Z)^2$, we have
\begin{equation*}
\xi(\bolda, \boldb) = \pm \frac{N^2}{3} \Eis^{0,0,1}(\bolda, \boldb).
\end{equation*}
\end{conjecture}

Conjectures \ref{conj xi1 Eis001} and \ref{conj xi Eis001} relate two motivic classes whose constructions are quite different. However, both classes are of modular nature, and we expect that there is a modular proof of the relation between them. See Section \ref{subsec: approach} for some speculation on this.

As further evidence for these conjectures, one could try to compute the $p$-adic regulator of $\xi(\bolda, \boldb)$. The $p$-adic regulator of the Beilinson elements has been computed by Gealy in terms of $p$-adic $L$-values \cite{Gea03}. Besser and De Jeu have obtained a general formula for the syntomic regulator on $K_4$ of curves using De Jeu's polylogarithmic complex \cite{BdJ12}.

\section{Perspectives} \label{sec: perspectives}

\subsection{Modular units} \label{subsec: uabcd}

The modular units $u(\bolda, \boldb, \boldc, \boldd)$ deserve further study. For $N=2$, there is only one such modular unit up to permutation of $\bolda, \boldb, \boldc, \boldd$, and we recover the modular $\lambda$ function:
\begin{equation*}
\lambda(\tau) = u((0,0), (1,0), (0,1), (1,1))(\tau).
\end{equation*}
Indeed, $\lambda$ is the cross-ratio of the $x$-coordinates of the $2$-torsion points of an elliptic curve.

For given $N$ and $a, b, c, d$ in $\Z/N\Z$, the divisor and the degree of the modular unit $u_1(a,b,c,d)$ on $X_1(N)$ can be computed using \eqref{eq u1 g0x} and Lemma \ref{lem div g0a}. Experimentally, these units have remarkably low degree. The following facts, obtained with PARI/GP, illustrate this.

\begin{enumerate}
\item A Hauptmodul for $X_1(N)$ is given by $u_1(0,1,2,3)$ for $N \in \{6,7,8\}$, and by $u_1(1,2,3,5)$ for $N \in \{9,10,12\}$ (these are the integers $N \geqslant 6$ such that $X_1(N)$ has genus 0). 
\item For $N$ prime, $7 < N < 300$, $N \neq 31$, the lowest degree among the units $u_1(a,b,c,d)$ is attained for $(a,b,c,d)=(1,2,3,5)$. Van Hoeij and Smith \cite{HS21} have proved that $\deg (u_1(1,2,3,5)) = [11N^2/840]$ for any prime $N>7$, where $[\cdot]$ denotes the nearest integer (in their notation, $u_1(1,2,3,5)$ is equal to the unit $F_7/F_8$ up to an homography). In fact, for $N>7$ prime, the unit $u_1(1,2,3,5)$ yields the lowest known degree for a non-constant map $X_1(N) \to \PP^1$ defined over $\Q$, except for $N \in \{31, 67, 101\}$, where the lowest known degree is one less \cite[Table 1]{DH14}.
\item For $N$ prime, $7 < N < 300$, the highest degree among the $u_1(a,b,c,d)$ is attained for $(a,b,c,d)=(0,1,3,4)$, the degree appearing to be $[N^2/35]$.
\end{enumerate}
These degrees can be compared to the Mason bound for $Y_1(N)$ (see Remark \ref{rk S-unit equation}), which equals $(N^2-1)/12$ if $N \geqslant 5$ is prime.

I do not know whether the units $u(\bolda, \boldb, \boldc, \boldd)$, resp. $u_1(a,b,c,d)$, span the group of modular units $\mathcal{O}(Y(N))^\times$, resp. $\mathcal{O}(Y_1(N))^\times$, up to constants. A maybe more natural question is whether the elements $\{u(\bolda, \boldb, \boldc, \boldd)\}_2$, resp. $\{u_1(a,b,c,d)\}_2$, span the vector spaces $B_2(Y(N))/B_2(\Q(\zeta_N))$, resp. $B_2(Y_1(N))/B_2(\Q)$, with the definition of \cite[Definition 2.1]{Gon08}.

It would be also interesting to investigate the integrality properties of these modular units, in the sense of an integral model of the modular curve.

\subsection{Constructing elements in $K_4$ of curves} \label{subsec: construct K4}

In Section \ref{subsec: de jeu curves}, we outlined a strategy to construct elements in $K_4$ of curves, using an auxiliary finite set $S$ of points on the curve. This approach requires to compute solutions to the $S$-unit equation for the function field of the curve. As we saw, this method works well for the modular curves $X(N)$ and $X_1(N)$, taking $S$ to be the set of cusps. The units $u(\bolda, \boldb, \boldc, \boldd)$ provide us with plenty of solutions to the $S$-unit equation (of the order of $N^8/64$, taking into account that permuting $\bolda, \boldb, \boldc, \boldd$ gives rise to only 6 distinct units). In this way, I was first able to construct cocycles on these curves for small values of $N$. I could also check that the associated $K_4$ elements were non-trivial, by computing numerically their images under the regulator map. However, this construction only worked case by case and, from a computational point of view, became impractical for larger $N$. Moreover, no pattern emerged in these cocycles. I then found a general and uniform construction, using the lift of the $3$-terms relations in $K_2(Y(N)) \otimes \Z\bigl[\frac{1}{6N}\bigr]$ from Section \ref{sec: 3-term}.

Although we concentrated in this article on modular curves, there are several other interesting situations. For example:
\begin{enumerate}
\item $X=\PP^1$ and $S = \{0,\infty\} \cup \mu_N$, where $\mu_N$ is the group of $N$th roots of unity (see the work of Zhao \cite{Zha22} for the $S$-unit equation in this context);
\item $X=E$ is an elliptic curve, and $S$ is a finite subgroup of $E$;
\item $X$ is the Fermat curve $F_N : x^N+y^N=z^N$, and $S = F_N \cap \{xyz = 0\}$.
\item $X$ is the Klein quartic $x^3 y + y^3 z + z^3 x = 0$ (which is isomorphic to a quotient of $F_7$, and whose base change to $\Q(\zeta_7)$ is isomorphic to $X(7)$ over $\Q(\zeta_7)$ \cite{Elk99}).
\end{enumerate}
It would be interesting to construct elements in $K_4$ of those curves.

\subsection{Comparison with the Beilinson elements} \label{subsec: approach}

To study Conjecture \ref{conj xi Eis001}, one might try to use the cohomology at infinite level explained in Section \ref{subsec: infinite level}. Let $\bolda, \boldb \in (\Q/\Z)^2$, and choose $N \geqslant 1$ such that $N\bolda = N\boldb = 0$. Consider the element
\begin{equation*}
\xi'(\bolda, \boldb) \in K_4^{(3)}(X(\infty)).
\end{equation*}
As explained in Section \ref{subsec: infinite level}, this element can conjecturally be computed using any level $N'$ divisible by $N$. Assuming $N'$ odd for simplicity, the element $\xi(\bolda, \boldb)$ at level $N'$ would then be the class of the cocycle
\begin{equation} \label{eq xi N'}
\Bigl( \sum_{\boldx \in (\Z/N'\Z)^2} \{u(\boldzero, \boldx, \bolda-\boldx, \boldb+\boldx)\}_2 \Bigr) \otimes \frac{g_\boldb}{g_\bolda}.
\end{equation}
One may view $(\Z/N'\Z)^2$ as the full $N'$-torsion subgroup of the universal elliptic curve $E(N)$ over $Y(N)$. Applying the regulator map and taking the limit as $N' \to \infty$, the sum \eqref{eq xi N'} becomes (formally) an integral along the fibres of $p : E(N) \to Y(N)$. This is reminiscent of the definition \eqref{def Eis001} of the Beilinson element $\Eis^{0,0,1}(\bolda, \boldb)$, which is obtained by pushing forward along $p$. Indeed, after applying the regulator map, the pushforward in motivic cohomology corresponds to integrating along the fibres of $p$.

\subsection{Applications to Mahler measures} \label{subsec: mahler}

We now describe the recent work \cite{Bru23}, which builds on the constructions of this article, and is also related to our running example. This work is concerned with the Mahler measure of the polynomial $(1+x)(1+y)+z$. Recall that the (logarithmic) Mahler measure $m(P)$ of a polynomial $P \in \C[x_1,\ldots,x_n]$ is defined by
\begin{equation*}
m(P) = \frac{1}{(2\pi i)^n} \int_{T^n} \log |P(x_1,\ldots,x_n)| \frac{dx_1}{x_1} \wedge \cdots \wedge \frac{dx_n}{x_n},
\end{equation*}
where $T^n$ is the $n$-torus $|x_1| = \cdots = |x_n| = 1$. We refer to \cite{BZ20} for an introduction to the Mahler measure and a survey of recent results.

Boyd and Rodriguez Villegas have conjectured relations between the Mahler measure of certain 3-variable polynomials and $L$-functions of elliptic curves evaluated at $s=3$ (see \cite[Section 8.4]{BZ20} and \cite{Lal15}). In \cite{Bru23}, we prove one of these identities.

\begin{theorem} \label{thm LE3}
We have $m((1+x)(1+y)+z) = -2L'(E,-1)$, where
\begin{equation*}
E : (1+x)^2 (1+y)^2 = xy
\end{equation*}
is the elliptic curve from our running example (see Example \ref{ex1}).
\end{theorem}

The proof of this theorem relies on the fact that $E$ is isomorphic to $X_1(15)$ and uses the elements $\xi_1(a,b)$ for $N=15$. This proof requires to have an explicit expression for the cocycles $\tilde{\xi}_1(a,b)$, so Theorem \ref{triangulation manin3} is crucial here.

In more detail, Lal\'{\i}n \cite{Lal15} has shown that
\begin{equation} \label{eq Lalin}
m((1+x)(1+y)+z) = \frac{1}{4\pi^2} \int_{\gamma_E^+} r_3(2)(\tilde{\xi}_E),
\end{equation}
where $\tilde{\xi}_E = \{-x\}_2 \otimes y - \{-y\}_2 \otimes x$ and $\gamma_E^+$ is a generator of $H_1(E(\C),\Z)^+$. By Example \ref{ex3}, the symbol $\tilde{\xi}_E$ defines an element $\xi_E$ of $K_4^{(3)}(E)$. But Example \ref{ex5} also gives an element $\xi_1(1,4)$ of $K_4^{(3)}(X_1(15))$. Both groups can be identified and are expected to be $1$-dimensional. Automated computations in the Goncharov complex reveal that
\begin{equation} \label{eq xiE}
\xi_E = -20 \xi_1(1,4).
\end{equation}
Theorem \ref{thm LE3} now follows from the theoretical computation of modular regulator integrals in \cite{BZ23}. Note that Theorem \ref{thm LE3} is consistent with the numerical computation in Example \ref{ex7}.

More generally, Lal\'{\i}n has shown that the Mahler measure of a polynomial $P(x_1, \ldots, x_n)$ is related to the Goncharov complexes of specific subvarieties of the hypersurface defined by $P$ \cite[Section 6]{Lal07}. In the example above, the subvariety is the so-called Maillot variety, defined by the equations $P(x,y,z) = P(\frac{1}{x}, \frac{1}{y}, \frac{1}{z}) = 0$. The relations \eqref{eq Lalin} and \eqref{eq xiE} justify our claim that the elements $\xi_1(a,b)$ are closer to the Mahler measure than the Beilinson elements. It would be extremely interesting to develop ``modular'' versions of the Goncharov complexes in arbitrary motivic weight and for higher-dimensional Kuga--Sato varieties, and to find modular constructions of cocycles in these complexes.

\vspace{\baselineskip}

\author{Fran\c{c}ois Brunault, francois.brunault@ens-lyon.fr\\
UMPA, \'{E}NS Lyon, 46 all\'{e}e d'Italie, 69007 Lyon, France}

\end{document}